\documentclass[a4paper,reqno]{amsart}

\usepackage[T1]{fontenc}
\usepackage[utf8x]{inputenc}
\usepackage[english]{babel}
\usepackage{yfonts}
\usepackage{dsfont}
\usepackage{amscd,amssymb,amsmath,amsthm,amsfonts}
\usepackage{mathtools}
\usepackage{accents}
\usepackage{tensor}
\usepackage{graphicx}
\usepackage{tikz}
\usepackage{tikz-cd}
\usetikzlibrary{calc,matrix,arrows,decorations.pathmorphing}
\usepackage{calc}
\usepackage[cal=boondox,scr=boondoxo]{mathalfa}
\usepackage{enumitem}
\usepackage{marginnote}
\usepackage{booktabs}
\usepackage{scalerel}[2016/12/29]
\usepackage{url}
\usepackage{contour}
\usepackage{ulem}

\usepackage{hyperref}
\hypersetup{colorlinks=true,pageanchor=false,
linkcolor=blue,citecolor=red,urlcolor=red}
\usepackage[all]{hypcap}

\allowdisplaybreaks

\newtheorem{theorem}{Theorem}[section]
\newtheorem{corollary}[theorem]{Corollary}
\newtheorem{proposition}[theorem]{Proposition}
\newtheorem{lemma}[theorem]{Lemma}

\theoremstyle{definition}

\theoremstyle{remark}
\newtheorem{remark}[theorem]{Remark}

\newcommand{\Z}{\mathbb{Z}}

\newcommand{\bbC}{\mathbb{C}}

\newcommand{\calC}{\mathcal{C}}

\newcommand{\calE}{\mathcal{E}}

\newcommand{\calH}{\mathcal{H}}

\newcommand{\calS}{\mathcal{S}}

\newcommand{\id}{\mathrm{id}}

\newcommand{\im}{\operatorname{im}}
\newcommand{\coker}{\operatorname{coker}}
\newcommand{\rank}{\operatorname{rank}}
\newcommand{\lk}{\operatorname{lk}}

\DeclareMathOperator{\ttimes}{\tilde{\times}}

\DeclareMathOperator{\bcs}{\natural}

\newcommand{\leqs}{\leqslant}
\newcommand{\geqs}{\geqslant}

\newcommand{\mods}[1]{\operatorname{\mathnormal{#1}-mod}}

\newcommand{\fsl}{\mathfrak{sl}}

\newcommand{\SO}{\mathrm{SO}}
\newcommand{\Ext}{\mathrm{Ext}}
\newcommand{\Hom}{\mathrm{Hom}}

\DeclareRobustCommand{\one}{\mathbin{\text{\includegraphics[height=\heightof{$\mathbf{1}$}]{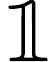}}}}

\makeatletter
\newcommand{\subalign}[1]{
  \vcenter{
    \Let@ \restore@math@cr \default@tag
    \baselineskip\fontdimen10 \scriptfont\tw@
    \advance\baselineskip\fontdimen12 \scriptfont\tw@
    \lineskip\thr@@\fontdimen8 \scriptfont\thr@@
    \lineskiplimit\lineskip
    \ialign{\hfil$\m@th\scriptstyle##$&$\m@th\scriptstyle{}##$\crcr
      #1\crcr
    }
  }
}
\makeatother

\def\clap#1{\hbox to 0pt{\hss#1\hss}}

\def\mathclap{\mathpalette\mathclapinternal}

\def\mathclapinternal#1#2{%
\clap{$\mathsurround=0pt#1{#2}$}}

\newcommand{\FRCob}{3\mathrm{Cob}^\sigma}

\newcommand{\RHB}{4\mathrm{HB}}

\newcommand{\KTan}{\mathrm{KTan}}
\newcommand{\Algf}{4\mathrm{Alg}}

\newcommand{\pic}[2][0]{\raisebox{-0.5\height + 2.5pt + #1pt}{\includegraphics{#2.pdf}}}

\newcommand\arxiv[2]{\href{https://arXiv.org/abs/#1}{\texttt{arXiv:\allowbreak #1} #2}}
\newcommand\doi[2]{\href{https://doi.org/#1}{#2}}

\contourlength{1pt}

\DeclareRobustCommand{\myuline}[1]{
 \ifmmode \text{\uline{$\phantom{#1}$}\llap{\contour{white}{$#1$}}}
 \else \uline{\phantom{#1}}\llap{\contour{white}{#1}} \fi
}

\newcommand{\dfrmtn}[1]{$#1$-de\-for\-ma\-tion}
\newcommand{\dfrmtns}[1]{$#1$-de\-for\-ma\-tions}
\newcommand{\dmnsnl}[1]{$#1$-di\-men\-sion\-al}

\newcommand{\hndlbd}[1]{$#1$-han\-dle\-bod\-y}
\newcommand{\hndlbds}[1]{$#1$-han\-dle\-bod\-ies}
\newcommand{\mnfld}[1]{$#1$-man\-i\-fold}
\newcommand{\mnflds}[1]{$#1$-man\-i\-folds}
\newcommand{\qvlnc}[1]{$#1$-e\-quiv\-a\-lence}
\newcommand{\qvlnt}[1]{$#1$-e\-quiv\-a\-lent}

\numberwithin{equation}{section}
  
\begin{document}

\raggedbottom

\title{Refined Bobtcheva--Messia Invariants of 4-Dimensional 2-Handlebodies}

\author[A. Beliakova]{Anna Beliakova} 
\address{Institute of Mathematics, University of Zurich, Winterthurerstrasse 190, CH-8057 Zurich, Switzerland} 
\email{anna@math.uzh.ch}

\author[M. De Renzi]{Marco De Renzi} 
\address{Institute of Mathematics, University of Zurich, Winterthurerstrasse 190, CH-8057 Zurich, Switzerland} 
\email{marco.derenzi@math.uzh.ch}



\begin{abstract}
 In this paper we refine our recently constructed invariants of \dmnsnl{4} \hndlbds{2} up to \dfrmtns{2}. More precisely, we define invariants of pairs of the form $(W,\omega)$, where $W$ is a \dmnsnl{4} \hndlbd{2}, $\omega$ is a relative cohomology class in $H^2(W,\partial W;G)$, and $G$ is an abelian group. The algebraic input required for this construction is a unimodular ribbon Hopf $G$-coalgebra. We study these refined invariants for the restricted quantum group $U = U_q \fsl_2$ at a root of unity $q$ of even order, and for its braided extension $\tilde{U} = \tilde{U}_q \fsl_2$, which fits in this framework for $G=\Z/2\Z$, and we relate them to our original invariant. We deduce decomposition formulas for the original invariants in terms of the refined ones, generalizing splittings of the Witten--Reshetikhin--Turaev invariants with respect to spin structures and cohomology classes. Moreover, we identify our non-refined invariant associated with the small quantum group $\bar{U} = \bar{U}_q \fsl_2$ at a root of unity $q$ whose order is divisible by $4$ with the refined one associated with the restricted quantum group $U$ for the trivial cohomology class $\omega=0$.
\end{abstract}

\maketitle

\section{Introduction}\label{S:introduction}

The central object of study in this paper are \dmnsnl{4} \hndlbds{2}\index{2-handlebody}, which are smooth \mnflds{4} obtained from the $4$-ball by attaching finitely many $1$-han\-dles and $2$-han\-dles. These handlebodies are usually represented by Kirby diagrams\index{Kirby diagram}, with dotted and undotted components corresponding to the attachment of $1$-han\-dles and $2$-han\-dles, respectively. We do not consider these \mnflds{4} up to diffeomorphism, but rather up to a more subtle equivalence relation called \textit{\dfrmtn{2}}\index{2-deformation}, or \textit{\qvlnc{2}}\index{2-equivalence}. Recall that, if a pair of \dmnsnl{4} handlebodies is diffeomorphic, then their handle decompositions can be related by a finite sequence of isotopies of attaching maps, handle slides, and creation/removal of canceling pairs of handles of index $0/1$, $1/2$, $2/3$, and $3/4$. By definition, \dfrmtns{2} are those diffeomorphisms implemented by sequences of handle moves that never introduce $3$-han\-dles and $4$-han\-dles, that is, creation/removal of $2/3$ and $3/4$ pairs are forbidden. It remains open whether \dfrmtns{2} form a proper subclass of the class of diffeomorphisms. This question is closely related to a deep open problem in combinatorial group theory, the Andrew--Curtis conjecture\index{Andrews-Curtis conjecture}. 

In \cite{BD21} we defined a wide class of new invariants of \dmnsnl{4} \hndlbds{2} up to \dfrmtns{2}. More precisely, starting from any unimodular ribbon category\index{ribbon category}\index{ribbon category!unimodular} $\calC$, we constructed a braided monoidal functor $J_4$, with source the category $\RHB$ of \dmnsnl{4} \hndlbds{2}, and target $\calC$. This can be considered as a close relative\footnote{In $\RHB$, morphisms are not cobordisms, the monoidal structure is not induced by disjoint union, and the braiding is not symmetric.} of a \dmnsnl{4} \textit{TQFT}\index{TQFT} (short for \textit{Topological Quantum Field Theory}). By a result of Bobtcheva and Piergallini \cite{BP11}, $\RHB$ is freely generated by a single \textit{BPH algebra}\index{Hopf algebra}\index{Hopf algebra!Bobtcheva-Piergallini} (short for \textit{Bobtcheva--Piergallini Hopf algebra}), the solid torus. For a unimodular ribbon category $\calC$, we proved that the end 
\[
 \calE = \int_{X \in \calC} X \otimes X^*
\]
is always a BPH algebra in $\calC$, and therefore the assignment of $\calE$ to the generator of $\RHB$ can always be extended to a functor $J_4 : \RHB \to \calC$, see \cite[Theorem~1.1]{BD21}. Notice that $\RHB$ comes equipped with a natural boundary functor $\partial$ with target the category $\FRCob$ of connected \dmnsnl{3} cobordisms with connected boundary (and signature). This category (or rather its quotient obtained by forgetting the signature) was intensively studied by Crane and Yetter, who realized that the punctured torus admits the structure of a braided Hopf algebra\index{Hopf algebra!braided} \cite{CY94}, by Kerler, who found a set of generating morphisms and a long list of elegant and conceptual relations among them \cite{K01}, and by Habiro, who announced a complete algebraic presentation \cite{A11}. The functor $\partial : \RHB \to \FRCob$ is compatible with the algebraic structure discovered in \cite{BP11}, and sends the BPH algebra generator of $\RHB$ (the solid torus) to the one of $\FRCob$ (the punctured torus), whose algebraic structure is a quotient of the previous one by a single additional relation (which defines a \textit{factorizable} BPH algebra).

When $\calC$ is factorizable (that is, when it has no non-trivial transparent objects, or equivalently when the Hopf copairing on the end is non-degenerate), we showed that $J_4$ factors as $J_3^\sigma \circ \partial$ for a functor $J_3^\sigma : \FRCob \to \calC$, see \cite[Theorem~1.2]{BD21}. This explains why, in all previous attempts at generalizing quantum invariants of \mnflds{3} (based on factorizable ribbon categories\index{ribbon category!factorizable}) to \mnflds{4}, the result always depended exclusively on \dmnsnl{3} boundaries (together with the signature, since $J_3^\sigma$ is affected by the usual framing anomaly).

In order to detect \dmnsnl{4} \hndlbds{2} that are diffeomorphic but not equivalent under 2-deformations, the functor $J_4$ needs to vanish against $S^2 \times D^2$. Indeed, our invariant is multiplicative under boundary connected sum, and a finite number of stabilizations by $S^2 \times D^2$ is sufficient to turn every pair of diffeomorphic \dmnsnl{4} \hndlbds{2} into a \qvlnt{2} pair. Setting $\calC = \mods{H}$ for a unimodular ribbon Hopf algebra\index{Hopf algebra!unimodular}\index{Hopf algebra!ribbon} $H$, we obtained that $J_4 (S^2 \times D^2)$ is an invertible scalar if and only if $H$ is cosemisimple (meaning $H^*$ is semisimple). By observing that the small quantum group\index{quantum sl(2)}\index{quantum sl(2)!small} $\bar{U} = \bar{U}_q \fsl_2$ at a root of unity $q$ of order $r \equiv 0 \pmod 4$ is neither factorizable nor (co)semisimple, we concluded that the corresponding invariant has the potential of detecting interesting \dmnsnl{4} phenomena.

This paper is devoted to the study of these invariants. More precisely, we investigate the scalars assigned by $J_4$ to endomorphisms of the tensor unit (the $3$-ball) in $\RHB$ in the case $\calC = \mods{H}$. In order to simplify the notation, we will denote the resulting invariant, which was first constructed by Bobtcheva and Messia \cite{BM02}, by $J_H$. Notice that, in this case, the end $\calE$ can be explicitly described as the adjoint representation of $H$, which determines a BPH algebra in $\mods{H}$ denoted $\myuline{H}$, and called the \textit{transmutation}\index{transmutation} of $H$. As a first step, we extend the definition of $J_H$ by allowing $H$ to be a unimodular ribbon Hopf $G$-coalgebra\index{group coalgebra}, see Section~\ref{S:group-coalgebras}. In this case, we construct an invariant of pairs $(W,\omega)$, where $W$ is a \dmnsnl{4} \hndlbd{2}, and $\omega$ is a relative cohomology class in $H^2(W,\partial W;G)$, see Theorem~\ref{T:main}.

As the main motivating example for this construction, we discuss in detail the restricted quantum group\index{quantum sl(2)!restricted} $U = U_q \fsl_2$ at a root of unity $q$ of even order $2p$, which is a unimodular Hopf $\Z/2\Z$-coalgebra that contains the small quantum group $\bar{U}$ as its degree zero part. The former is not ribbon, but admits a ribbon extension $\tilde{U} = \tilde{U}_q \fsl_2$. However, $U$ is factorizable, while $\tilde{U}$ is not. The adjoint representation of $U$ is closed under the adjoint action of $\tilde{U}$, and thus determines an object in $\mods{\tilde{U}}$. This object admits a transmutation, denoted $\myuline{U}$, which provides a BPH algebra in $\mods{\tilde{U}}$, as established in Proposition~\ref{P:3-modular}. Thanks to \cite[Theorem~1.2]{BD21}, the corresponding invariant, denoted $J_U$, only depends on the \dmnsnl{3} boundary and signature of \dmnsnl{4} \hndlbds{2}, see Corollary~\ref{C:U_factorizable}. Furthermore, the refined invariant associated with $\tilde{U}$ can be actually computed entirely inside $U$, as proved in Proposition~\ref{P:restricted}, and deserves therefore to be denoted $J_U$ too.

In Section~\ref{S:decompositions} we establish a few decomposition formulas for both the refined and the non-refined invariant $J_U$. The precise form of these decomposition formulas depends on the arithmetic properties of the order $2p$ of the root of unity $q$. When $p \equiv 0 \pmod 4$, we show that $U$ can be used to define an invariant of \mnflds{3} equipped with spin structures\index{spin structure}, which is always denoted $J_U$, and that
\begin{equation*}
 J_U(W,\omega) = \lambda(v_-)^{\sigma(W)} \sum_{s \in \calS(\partial W,\omega)} J_U(\partial W,s),
\end{equation*}
where $\calS(\partial W,\omega)$ is the set of spin structures $s$ on $\partial W$ with second relative Stiefel--Whitney class $w_2(W,s) = \omega$, where $\sigma(W)$ is the signature of $W$, and where $\lambda$ and $v_-$ are the left integral and the inverse ribbon element of $U$ respectively. For the non-refined invariant, we obtain
\begin{equation*}
 J_U(W) = \lambda(v_-)^{\sigma(W)} \sum_{s \in \calS(\partial W)} J_U(\partial W,s),
\end{equation*}
where $\calS(\partial W)$ is the set of all spin structures $s$ on $\partial W$. This result generalizes the well-known decomposition of the Witten--Reshetikhin--Turaev\index{Witten-Reshetikhin-Turaev invariant} (WRT) invariants in terms of spin structures \cite{B92, B98}. When $p \equiv 2 \pmod 4$, the picture is similar, although this time $U$ can be used to define an invariant of \mnflds{3} equipped with first cohomology classes with $\Z/2\Z$-coefficients, and
\begin{equation*}
 J_U(W,\omega) = \lambda(v_-)^{\sigma(W)} \sum_{\varphi \in \calH(\partial W,\omega)} J_U(\partial W,\varphi),
\end{equation*}
where $\calH(\partial W,\omega)$ is the set of cohomology classes $\varphi$ in $H^1(\partial W;\Z/2\Z)$ satisfying $\delta^*(\varphi) = \omega$ for the coboundary homomorphism 
\[
 \delta^* : H^1(\partial W;\Z/2\Z) \to H^2(W,\partial W;\Z/2\Z)
\]
coming from the long exact sequence of the pair $(W,\partial W)$ in cohomology with $\Z/2\Z$-coefficients. Again, for the non-refined invariant, we obtain
\begin{equation*}
 J_U(W) = \lambda(v_-)^{\sigma(W)} \sum_{\varphi \in \calH(\partial W)} J_U(\partial W,\varphi),
\end{equation*}
where $\calH(\partial W) = H^1(\partial W;\Z/2\Z)$, compare with \cite{B98} for the analogous decomposition of the WRT invariants. This is proved in Theorem~\ref{T:decomposition}, see also Remark~\ref{R:to_be_precise}.

In addition, we show in Proposition~\ref{P:restricted_vs_small} that, when the cohomology class is taken to be $\omega = 0$, the refined invariant $J_U$ associated with the restricted quantum group $U$ recovers the non-refined invariant $J_{\bar{U}}$ associated with the small quantum group $\bar{U}$. In other words, we have
\begin{equation*}
 J_{\bar{U}}(W) = J_{U}(W,0).
\end{equation*}
Using our first decomposition formula, we deduce that, for $p \equiv 0 \pmod 4$, and for $W$ a non-spin \mnfld{4}, we have
\begin{equation*}
 J_{\bar{U}}(W) = 0,
\end{equation*}
because in this case $\calS(\partial W,0) = \varnothing$, and so $J_U(W,0)=0$.
 
These results imply that the scalar invariant $J_{\bar{U}}$ essentially depends on \dmnsnl{3} boundaries, and cannot detect any truly \dmnsnl{4} phenomenon beyond the signature, the Euler characteristic (see Appendix~\ref{A:rescaling}), and the spin status of \dmnsnl{4} \hndlbds{2}. This happens because the small quantum group $\bar{U}$ is a finite index Hopf subalgebra of a factorizable Hopf algebra\index{Hopf algebra!factorizable}, the restricted quantum group $U$. We expect the invariant to behave similarly for other quantum groups, such as those discussed in \cite{LO16}. We remark however two interesting possibilities going forward: first of all, the functor $J_4 : \RHB \to \mods{\bar{U}}$ seems to be more sensitive and rich than its underlying scalar invariant, as it truly manages to distinguish $1$-han\-dles from $2$-han\-dles in the non-factorizable case. Furthermore, the question of finding non-trivial central elements that could be used to deform the scalar invariant $J_{\bar{U}}$ as in \cite[Theorem~2.14]{BM02} remains wide open. More generally, the same holds in the quest for other BPH algebras that could be sensitive to \dmnsnl{4} topology.

\subsection*{Acknowledgments}

Our work was supported by National Center of Competence in Research (NCCR) SwissMAP, and by Grant n.~200020\_207374 of the Swiss National Science Foundation (SNSF).

\section{Hopf group-coalgebras}\label{S:group-coalgebras}

In this section, we recall the definition of Hopf group-coalgebras, as introduced by Turaev \cite{T10}. Let $G$ be an abelian group. Following \cite{V00,GHP20}, a \textit{Hopf $G$-coalgebra}\index{group coalgebra} is a family $H = \{ H_\alpha \mid \alpha \in G \}$ of vector spaces over a field $\Bbbk$ equipped with:
\begin{itemize}
 \item a family $\mu = \{ \mu_\alpha : H_\alpha \otimes H_\alpha \to H_\alpha \mid \alpha \in G \}$ of \textit{products};
 \item a family $\eta = \{ \eta_\alpha : \Bbbk \to H_\alpha \mid \alpha \in G \}$ of \textit{units};
 \item a family $\Delta = \{ \Delta_{\alpha,\beta} : H_{\alpha+\beta} \to H_\alpha \otimes H_\beta \mid \alpha,\beta \in G \}$ of \textit{coproducts};
 \item a \textit{counit} $\varepsilon_0 : H_0 \to \Bbbk$;
 \item a family $S = \{ S_\alpha : H_\alpha \to H_{-\alpha} \mid \alpha \in G \}$ of \textit{antipodes}.
\end{itemize}
These data satisfy:
\begin{enumerate}
 \item $\mu_\alpha \circ (\mu_\alpha \otimes \id_{H_\alpha}) = \mu_\alpha \circ (\id_{H_\alpha} \otimes \mu_\alpha)$ for every $\alpha \in G$;
 \item $\mu_\alpha \circ (\eta_\alpha \otimes \id_{H_\alpha}) = \id_{H_\alpha} = \mu_\alpha \circ (\id_{H_\alpha} \otimes \eta_\alpha)$ for every $\alpha \in G$;
 \item $(\Delta_{\alpha,\beta} \otimes \id_{H_{\gamma}}) \circ \Delta_{\alpha+\beta,\gamma} = (\id_{H_\alpha} \otimes \Delta_{\beta,\gamma}) \circ \Delta_{\alpha,\beta+\gamma}$ for all $\alpha,\beta,\gamma \in G$;
 \item $(\varepsilon_0 \otimes \id_{H_\alpha}) \circ \Delta_{0,\alpha} = \id_{H_\alpha} = (\id_{H_\alpha} \otimes \varepsilon_0) \circ \Delta_{\alpha,0}$ for every $\alpha \in G$;
 \item $\Delta_{\alpha,\beta} \circ \mu_{\alpha+\beta} = (\mu_\alpha \otimes \mu_\beta) \circ (\id_{H_\alpha} \otimes \tau_{H_\beta,H_\alpha} \otimes \id_{H_\beta}) \circ (\Delta_{\alpha,\beta} \otimes \Delta_{\alpha,\beta})$ for all $\alpha,\beta \in G$, where $\tau_{H_\beta,H_\alpha} : H_\beta \otimes H_\alpha \to H_\alpha \otimes H_\beta$ is the standard transposition;
 \item $\varepsilon_0 \circ \mu_0 = \varepsilon_0 \otimes \varepsilon_0$;
 \item $\Delta_{\alpha,\beta} \circ \eta_{\alpha+\beta} = \eta_\alpha \otimes \eta_\beta$ for all $\alpha,\beta \in G$;
 \item $\varepsilon_0 \circ \eta_0 = \id_\Bbbk$;
 \item $\mu_\alpha \circ (S_{-\alpha} \otimes \id_{H_\alpha}) \circ \Delta_{-\alpha,\alpha} = \eta_\alpha \circ \varepsilon_0 = \mu_\alpha \circ (\id_{H_\alpha} \otimes S_{-\alpha}) \circ \Delta_{\alpha,-\alpha}$ for every $\alpha \in G$.
\end{enumerate}
We will often use the shorthand notation
\[
 \mu_\alpha(x \otimes y) = xy, \qquad 
 \eta_\alpha(1) = 1_\alpha, \qquad
 \Delta_{\alpha,\beta}(x) = x_{(1,\alpha)} \otimes x_{(2,\beta)}.
\]
Remark that $H_\alpha$ is a unital associative algebra for every $\alpha \in G$, and that $H_0$ is a Hopf algebra.\index{Hopf algebra}

A \textit{ribbon Hopf $G$-coalgebra} is a Hopf $G$-coalgebra $H = \{ H_\alpha \mid \alpha \in G \}$ equipped with:
\begin{itemize}
 \item a family $R = \{ R_{\alpha,\beta} = R'_\alpha \otimes R''_\beta \in H_\alpha \otimes H_\beta \mid \alpha,\beta \in G \}$ of \textit{R-matrices};
 \item a family $v_+ = \{ v_\alpha \in H_\alpha \mid \alpha \in G \}$ of central invertible \textit{ribbon elements}.
 \end{itemize}
These data satisfy:
\begin{enumerate}
 \item $R'_\alpha x_{(1,\alpha)} \otimes R''_\beta x_{(2,\beta)} = x_{(2,\alpha)} R'_\alpha \otimes x_{(1,\beta)} R''_\beta$ for all $\alpha,\beta \in G$ and $x \in H_{\alpha+\beta}$;
 \item $R'_\alpha \otimes \Delta_{\beta,\gamma}(R''_{\beta+\gamma}) = (R'_\alpha \otimes 1_\beta \otimes R''_\gamma)(R'_\alpha \otimes R''_\beta \otimes 1_\gamma)$ for all 
 $\alpha,\beta,\gamma \in G$;
 \item $\Delta_{\alpha,\beta}(R'_{\alpha+\beta}) \otimes R''_\gamma = (R'_\alpha \otimes 1_\beta \otimes R''_\gamma)(1_\alpha \otimes R'_\beta \otimes R''_\gamma)$ for all
 $\alpha,\beta,\gamma \in G$;
 \item $v_\alpha^2 = u_\alpha S_{-\alpha}(u_{-\alpha})$ for all $\alpha \in G$, where
 $u_\alpha := S_{-\alpha}(R''_{-\alpha}) R'_\alpha$;
 \item $\Delta_{\alpha,\beta}(v_{\alpha+\beta}) = (v_\alpha \otimes v_\beta) (S_{-\alpha}(R'_{-\alpha}) \otimes R''_\beta)(R''_\alpha \otimes S_{-\beta}(R'_{-\beta}))$ for all $\alpha,\beta \in G$;
 \item $\varepsilon_0(v_0) = 1$;
 \item $S_\alpha(v_\alpha) = v_{-\alpha}$ for every $\alpha \in G$.
\end{enumerate}
Remark that $H_0$ is a ribbon Hopf algebra\index{Hopf algebra!ribbon}, and that the family 
\[
 u = \{ u_\alpha \in H_\alpha \mid \alpha \in G \}
\]
of \textit{Drinfeld elements} determines a family 
\[
 g = \{ g_\alpha \in H_\alpha \mid \alpha \in G \}
\]
of \textit{pivotal elements} 
\[
 g_\alpha := u_\alpha v_\alpha^{-1}.
\]

A \textit{unimodular Hopf $G$-coalgebra} is a Hopf $G$-coalgebra $H = \{ H_\alpha \mid \alpha \in G \}$ equipped with:
\begin{itemize}
 \item a family $\lambda = \{ \lambda_\alpha : H_\alpha \to \Bbbk \mid \alpha \in G \}$ of \textit{left integrals};
 \item a \textit{two-sided cointegral} $\Lambda_0 \in H_0$.
\end{itemize}
These data satisfy:
\begin{enumerate}
 \item $\lambda_\beta(x_{(2,\beta)}) x_{(1,\alpha)} = \lambda_{\alpha+\beta}(x) 1_\alpha$ for all $\alpha,\beta \in G$ and $x \in H_{\alpha+\beta}$;
 \item $x \Lambda_0 = \varepsilon_0(x) \Lambda_0 = \Lambda_0 x$ for every $x \in H_0$.
\end{enumerate}
Remark that $H_0$ is a unimodular Hopf algebra\index{Hopf algebra!unimodular}.

If $A$ is an algebra, recall that a \textit{central idempotent} of $A$ is an element $z \in A$ satisfying
\begin{align*}
 zx &= xz, &
 z^2 &= z
\end{align*}
for every $x \in A$. If $G$ is a finite abelian group, and $H$ is a Hopf algebra, then a family of central idempotents $\{ 1_\alpha \in H \mid \alpha \in G \}$ is called a \textit{$G$-splitting system} if it satisfies
\begin{align*}
 1_\alpha 1_\beta &= \delta_{\alpha,\beta} 1_\alpha, &
 1 &= \sum_{\alpha \in G} 1_\alpha, \\*
 \Delta(1_\alpha) &= \sum_{\beta \in G} 1_{\alpha-\beta} \otimes 1_\beta, &
 \varepsilon(1_\alpha) &= \delta_{\alpha,0}, \\*
 S(1_\alpha) &= 1_{-\alpha}.
\end{align*}
If $\{ 1_\alpha \in H \mid \alpha \in G \}$ is a $G$-splitting system in $H$ for a finite abelian group $G$, then $\{ H 1_\alpha \mid \alpha \in G \}$ is a Hopf $G$-coalgebra, with product
\[
 \mu_\alpha(x 1_\alpha \otimes y 1_\alpha) = xy 1_\alpha
\]
for every $\alpha \in G$, coproduct
\[
 \Delta_{\alpha,\beta}(x 1_{\alpha+\beta}) = x_{(1)} 1_\alpha \otimes x_{(2)} 1_\beta
\]
for all $\alpha,\beta \in G$, with counit
\[
 \varepsilon_0(x 1_0) = \varepsilon(x),
\]
and with antipode
\[
 S_\alpha(x 1_\alpha) = S(x) 1_{-\alpha}
\]
for every $\alpha \in G$. If $H$ is a ribbon Hopf algebra, then $\{ H 1_\alpha \mid \alpha \in G \}$ is a ribbon Hopf $G$-coalgebra, with R-matrix
\[
 R_{\alpha,\beta} = R(1_\alpha \otimes 1_\beta)
\]
for all $\alpha,\beta \in G$, and ribbon element
\[
 v_\alpha = v 1_\alpha
\]
for every $\alpha \in G$. If $H$ is a unimodular Hopf algebra, then $\{ H 1_\alpha \mid \alpha \in G \}$ is a unimodular Hopf $G$-coalgebra, with left integral
\[
 \lambda_\alpha(x 1_\alpha) = \lambda(x 1_\alpha)
\]
for every $\alpha \in G$, and two-sided cointegral
\[
 \Lambda_0 = \Lambda.
\]
Vice versa, if $\{ H_\alpha \mid \alpha \in G \}$ is a Hopf $G$-coalgebra for a finite abelian group $G$, then its \textit{direct sum}
\[
 H := \bigoplus_{\alpha \in G} H_\alpha
\]
is a Hopf algebra, with product
\[
 \mu \left( \bigoplus_{\alpha,\beta \in G} x_\alpha \otimes y_\beta \right) = \bigoplus_{\alpha \in G} x_\alpha y_\alpha,
\]
with unit
\[
 1 = \bigoplus_{\alpha \in G} 1_\alpha,
\]
with coproduct
\[
 \Delta \left( \bigoplus_{\alpha \in G} x_\alpha \right) = \bigoplus_{\alpha,\beta \in G} (x_{\alpha+\beta})_{(1,\alpha)} \otimes (x_{\alpha+\beta})_{(2,\beta)},
\]
with counit
\[
 \varepsilon \left( \bigoplus_{\alpha \in G} x_\alpha \right) = \varepsilon_0(x_0),
\]
and with antipode
\[
 S \left( \bigoplus_{\alpha \in G} x_\alpha \right) = \bigoplus_{\alpha \in G} S_{-\alpha}(x_{-\alpha}).
\]
If $\{ H_\alpha \mid \alpha \in G \}$ is a ribbon Hopf $G$-coalgebra, then $H$ is a ribbon Hopf algebra, with R-matrix
\[
 R = \bigoplus_{\alpha, \beta \in G} R_{\alpha,\beta},
\]
and ribbon element
\[
 v = \bigoplus_{\alpha \in G} v_\alpha.
\]
If $\{ H_\alpha \mid \alpha \in G \}$ is a unimodular Hopf $G$-coalgebra, then $H$ is a unimodular Hopf algebra, with left integral
\[
 \lambda \left( \bigoplus_{\alpha \in G} x_\alpha \right) = \lambda_\alpha(x_\alpha),
\]
and two-sided cointegral
\[
 \Lambda = \Lambda_0.
\]

For a finite abelian group $G$, we say a ribbon Hopf $G$-coalgebra $\{ H_\alpha \mid \alpha \in G \}$ is \textit{factorizable} if the direct sum 
\[
 H = \bigoplus_{\alpha \in G} H_\alpha
\]
is a factorizable ribbon Hopf algebra\index{Hopf algebra!factorizable}. By definition, this means that the \textit{Drinfeld map}
\begin{align*}
 D : H^* &\to H \\*
 f &\mapsto f(M'_+)M''_+
\end{align*}
is a linear isomorphism, where
\[
 M_+ = M'_+ \otimes M''_+ = (R'' \otimes R')(R' \otimes R'') \in H \otimes H
\]
is the \textit{M-matrix} associated with the R-matrix $R = R' \otimes R'' \in H \otimes H$.

\section{Refined invariant}

In this section, we fix an abelian group $G$ and a unimodular ribbon Hopf $G$-coalgebra $H$, and we define a refined invariant of \dmnsnl{4} \hndlbds{2} equipped with relative cohomology classes with coefficients in $G$. Let
\[
 \varnothing \subset W^0 \subset W^1 \subset W^2 = W
\]
be a \dmnsnl{4} \hndlbd{2} featuring a single $0$-han\-dle and a finite number of handles of index $1$ and $2$. By reversing the handle presentation, $W$ can be obtained from $\partial W \times I$ by attaching a finite number of handles of index $2$ and $3$, followed by a single $4$-han\-dle, which means
\[
 \partial W \times I \subset \tilde{W}^2 \subset \tilde{W}^3 \subset \tilde{W}^4 = W.
\]
Now let $\omega \in H^2(W,\partial W;G)$ be a cohomology class. Since $H_1(W,\partial W) = 0$, we have
\begin{align*}
 H^2(W,\partial W;G) 
 &\cong \Hom(H_2(W,\partial W),G) \oplus \Ext(H_1(W,\partial W),G) \\*
 &\cong \Hom(H_2(W,\partial W),G),
\end{align*}
thanks to the universal coefficient theorem\footnote{Whenever coefficients are omitted, they are assumed to be in $\Z$.}. Notice that the space $C_k(W,\partial W)$ of relative cellular $k$-chains of $(W,\partial W)$ is generated by the set of cocores $\{ 0 \} \times D^k$ of $(4-k)$-han\-dles $D^{4-k} \times D^k$. If $D$ is a Kirby diagram of $W$, and if we fix an arbitrary ordering of all components of $D$, then, up to isotopy, the cocore of the $i$th $1$-han\-dle appears in the Kirby diagram of $W$ as the \dmnsnl{2} Seifert disc $d_i$ of the $i$th dotted component (the complementary hemisphere of $d_i$ in the boundary of the cocore lies in $\partial W$, and has therefore been carved under the diagram, while the interior of the cocore lies in the interior of $W$). Similarly, if we fix an arbitrary orientation for every undotted component of $D$, then, up to isotopy, the cocore of the $j$th $2$-han\-dle appears in the Kirby diagram of $W$ as the \dmnsnl{2} Seifert disc $m_j$ of a positive meridian of the $j$th undotted component. The differential $\partial : C_3(W,\partial W) \to C_2(W,\partial W)$ is defined as $\partial(d_i) = m_{j_1} + \ldots + m_{j_k}$, where $d_i$ and $m_{j_1}, \ldots, m_{j_k}$ are related by the following picture.
\[
 \pic{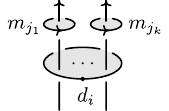}
\]
In other words, in the specified bases, $\partial$ is represented by the linking matrix between dotted and undotted components. If the orientation of one of the vertical strands in the above configuration is reversed, the sign of the contribution of the corresponding meridian to the sum should also be reversed. Then, $H_2(W,\partial W)$ is by definition equal to $\coker(\partial)$, and $H^2(W,\partial W;G)$ can be identified with the subgroup of $\Hom(C_2(W,\partial W),G)$ composed of those linear maps that vanish on $\im(\partial)$.

A \textit{$G$-Kirby diagram} of $(W,\omega)$ is obtained from a Kirby diagram of $W$ by orienting every undotted component and labeling it by $\omega(m) \in G$, where $m \in C_2(W,\partial W)$ denotes the \dmnsnl{2} Seifert disc of a positive meridian of the undotted component. Two $G$-Kirby diagrams represent the same pair $(W,\omega)$ if and only if they are related by a finite sequence of \textit{$G$-Kirby moves}:
\begin{align*}
 \pic{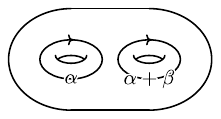} &\leftrightsquigarrow \pic{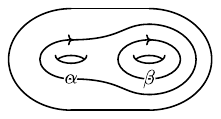} \tag{GK1}\label{E:GK1} \\*
 \pic{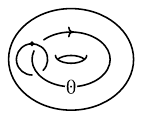} &\leftrightsquigarrow \pic{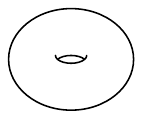} \tag{GK2}\label{E:GK2} \\*
 \pic{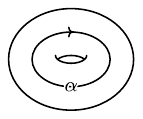} &\leftrightsquigarrow \pic{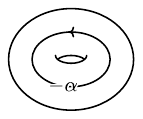} \tag{GK3}\label{E:GK3}
\end{align*}
Pictures~\eqref{E:GK1}--\eqref{E:GK3} represent operations performed inside \dmnsnl{3} \hndlbds{1} embedded into $S^3$, and they leave $G$-Kirby diagrams unchanged in the complement.

In order to define an invariant of the pair $(X,\omega)$, let us consider a $G$-Kirby diagram $D$ representing $(X,\omega)$. First of all, we insert beads labeled by components of the R-matrix around crossings as shown:
\begin{align*}
 \pic{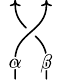} &\mapsto \pic{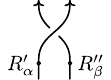} &
 \pic{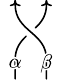} &\mapsto \pic{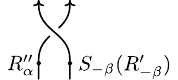}
\end{align*}
If a crossing is obtained from the ones above by reversing the orientation of a strand, the label of the corresponding bead is evaluated against the antipode. For instance, if the orientation of the strand labeled by $\alpha$ in the left-most crossing above is reversed, the corresponding bead changes from $R'_\alpha$ to $S_{-\alpha}(R'_{-\alpha})$. Next, we insert beads labeled by the pivotal element around right-oriented extrema as shown:
\begin{align*}
 \pic{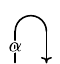} &\mapsto \pic{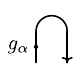} &
 \pic{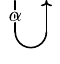} &\mapsto \pic{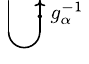}
\end{align*}
This operation leaves left-oriented extrema untouched. Then, we remove dotted components, while also inserting beads labeled by coproducts of the cointegral as shown:
\begin{align*}
 \pic{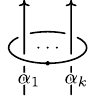} &\mapsto \pic{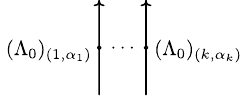}
\end{align*}
If the orientation of one of the strands piercing a Seifert disc for the dotted unknot is reversed, the label of the corresponding bead is evaluated against the antipode. For instance, if the orientation of the strand labeled by $\alpha_i$ above is reversed, the corresponding bead changes from $(\Lambda_0)_{(i,\alpha_i)}$ to $S_{-\alpha_i}((\Lambda_0)_{i,-\alpha_i})$. When $k = 0$, removing a dotted component costs a multiplicative factor of $\varepsilon_0(\Lambda_0)$ in front of $D$. Next, we collect all beads sitting on the same component in one place, and we multiply everything together according to the rule
\begin{align*}
 \pic{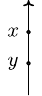} &= \pic{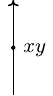}
\end{align*}
In the end, we are left with a decorated diagram of the form
\begin{equation}
 B(D) = \pic{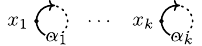} \label{E:bead_presentation}
\end{equation}
We say a label 
\[
 x_1 \otimes \ldots \otimes x_k \in H_{\alpha_1} \otimes \ldots \otimes H_{\alpha_k}
\]
obtained this way is a \textit{bead presentation} of $D$.

\begin{theorem}\label{T:main}
 If $W$ is a \dmnsnl{4} \hndlbd{2} equipped with a relative cohomology class $\omega \in H^2(W,\partial W;G)$, if $D$ is a $G$-Kirby diagram representing $(W,\omega)$, and if $x_1 \otimes \ldots \otimes x_k \in H_{\alpha_1} \otimes \ldots \otimes H_{\alpha_k}$ is a bead presentation of $D$, then the scalar
 \[
  J_H(W,\omega) = \prod_{i=1}^k \lambda_{\alpha_i}(x_i g_{\alpha_i}^{-1})
 \]
 is a topological invariant of the pair $(W,\omega)$.
\end{theorem}

\begin{proof}
 We need to check that $J_H(W,\omega)$ is invariant under $G$-Kirby moves. For what concerns move \eqref{E:GK1}, we have
 \[
  \pic{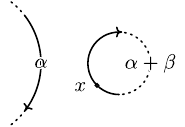} \rightsquigarrow \pic{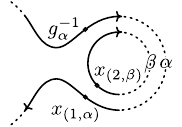}
 \]
 Then the claim follows from
 \begin{align*}
  \lambda_\beta(x_{(2,\beta)} g_\beta^{-1}) x_{(1,\alpha)} g_\alpha^{-1}
  &= \lambda_{\alpha+\beta}(x g_{\alpha+\beta}^{-1}) 1_\alpha.
 \end{align*}
 Similarly, for what concerns move \eqref{E:GK2}, we have
 \[
  \pic{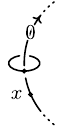} \rightsquigarrow \pic{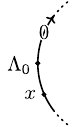}
 \]
 Then the claim follows from
 \begin{align*}
  \lambda_0(\Lambda_0 x) &= \varepsilon_0(x),
 \end{align*}
 together with the fact that
 \begin{align*}
  \varepsilon_0(R'_0) R''_\alpha &= 1_\alpha = \varepsilon_0(R''_0) R'_\alpha, &
  \varepsilon_0(g_0) &= 1, &
  \varepsilon(x_{(1,0)}) x_{(2,\alpha)} &= x = \varepsilon_0(x_{(2,0)}) x_{(1,\alpha)}.
 \end{align*}
 Finally, for what concerns move \eqref{E:GK3}, we have
 \[
  \pic{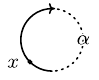} \rightsquigarrow \pic{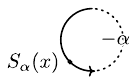}
 \]
 In order to see this, it is useful to represent the component as the closure of a braid. Then the claim follows from
 \begin{align*}
  \lambda_{-\alpha}(S_\alpha(x) g_{-\alpha}^{-1}) 
  &= \lambda_{-\alpha}(S_\alpha(g_\alpha x))
  = \lambda_\alpha(g_\alpha^{-1} x)
  = \lambda_\alpha(x g_\alpha^{-1}),
 \end{align*}
 where the second and third equalities are a consequence of \cite[Theorem~4.2]{V00}, which implies that
 \begin{align*}
  \lambda_{-\alpha}(S_\alpha(x)) &= \lambda_\alpha(g_\alpha^{-2} x), &
  \lambda_\alpha(xy) &= \lambda_\alpha(yS_{-\alpha}(S_\alpha(x)))
 \end{align*}
 for all $x,y \in H_\alpha$. Indeed, the distinguished group-like element is $g_\alpha^2 \in H_\alpha$ in our notation, and we are considering a left integral instead of a right integral.
\end{proof}

\begin{remark}
 When $G = \Z/2\Z$, then $\alpha = -\alpha$ for all $\alpha \in G$, and the orientations can be dropped. In this case, the algorithm given above, which follows the oriented approach of Hennings \cite{H96}, has to be replaced with the equivalent algorithm adapted from the unoriented approach of Kauffman and Radford \cite{KR94}, which is the one followed in \cite[Section~8.1]{BD21}.
\end{remark}

\section{General decomposition formulas}

We say a \dmnsnl{4} \hndlbd{2} is \textit{geometrically simply connected} if it admits a handle decomposition without $1$-han\-dles and $k$-han\-dles for $k > 2$. We will show that, if $G$ is a finite abelian group, and $H$ is a factorizable ribbon Hopf $G$-coalgebra, then both the refined and the non-refined invariants can be computed as sums of invariants of geometrically simply connected \dmnsnl{4} \hndlbds{2}. In other words, we can always trade $1$-han\-dles for $2$-han\-dles. In order to do this, let us consider a connected \dmnsnl{4} \hndlbd{2} $W$ with boundary $\partial W = M$, let $D$ be a Kirby diagram for $W$, let $L$ be the framed link obtained from $D$ by trading $1$-han\-dles for $2$-han\-dles (that is, by erasing dots), and let $E$ be the geometrically simply connected \dmnsnl{4} \hndlbd{2} represented by $L$. Notice that 
\[
 \partial E \cong \partial W = M,
\]
that $L$ provides a surgery presentation for $M$. We stress the fact that both $L$ and $E$ depend crucially on $D$, and that a different choice of diagram representing $W$ would yield different results.

Notice that we have a natural inclusion $\iota : C_2(W,M) \hookrightarrow C_2(E,M)$. Furthermore, as we explained before, relative cohomology classes in $H^2(W,M)$ can be identified with linear maps in $\Hom(C_2(W,M),G)$ vanishing of the image of the differential $\partial : C_3(W,M) \to C_2(W,M)$, and similarly relative cohomology classes in $H^2(E,M)$ can be identified with linear maps in $\Hom(C_2(E,M),G)$, since $C_3(E,M) = 0$. Then, let us set 
\begin{align*}
 \calE(E) &:= H^2(E,M;G), \\*
 \calE(E,\omega) &:= \{ \psi \in \calE(E) \mid \psi \circ \iota = \omega \}.
\end{align*}

\begin{proposition}\label{P:general_decomposition}
 If $G$ is a finite abelian group, and $H$ is a factorizable ribbon Hopf $G$-coalgebra, let $W$ be a connected \dmnsnl{4} \hndlbd{2} with boundary $\partial W = M$, let $D$ be a Kirby diagram for $W$, let $L$ be the framed link obtained from $D$ by trading $1$-han\-dles for $2$-han\-dles, and let $E$ be the \dmnsnl{4} \hndlbd{2} represented by $L$. Then, for every relative cohomology class $\omega \in H^2(W,M;G)$ we have
 \begin{align}
  J_H(W,\omega) &= \sum_{\psi \in \calE(E,\omega)} J_H(E,\psi), \label{E:general_decomposition_omega}
 \end{align}
 and for the non-refined invariant we have
 \begin{align}
  J_H(W) &= \sum_{\psi \in \calE(E)} J_H(E,\psi). \label{E:general_decomposition}
 \end{align}
\end{proposition}

\begin{proof}
 Let us start by proving Equation~\eqref{E:general_decomposition_omega}. To this end, let us show that every $1$-han\-dle can be traded for a $2$-han\-dle, provided we take into account all possible ways of extending $\omega$ to the newly created $2$-han\-dle. This amounts to showing that, if $\alpha_1,\ldots,\alpha_k \in G$ satisfy
 \[
  \sum_{j=1}^k \alpha_j = 0,
 \]
 then a local replacement in the $G$-Kirby diagram $D$ of the form
 \[
  \pic{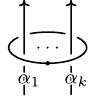} \rightsquigarrow \sum_{\beta \in G} \pic{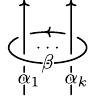}
 \]
 does not change the invariant (provided we extend it linearly to linear combinations of $G$-Kirby diagrams). In order to prove this, it is sufficient to notice that factorizability of $H$ implies
 \begin{align*}
  &\sum_{\beta \in G} \lambda_\beta(S(M'_+) 1_\beta) (M''_+)_{(1)} 1_{\alpha_1}  \otimes \ldots \otimes (M''_+)_{(k)} 1_{\alpha_k} \\*
  &\hspace*{\parindent} = \lambda(S(M'_+)) (M''_+)_{(1)} 1_{\alpha_1}  \otimes \ldots \otimes (M''_+)_{(k)} 1_{\alpha_k} \\*
  &\hspace*{\parindent} = \Lambda_{(1)} 1_{\alpha_1}  \otimes \ldots \otimes \Lambda_{(k)} 1_{\alpha_k},
 \end{align*}
 where $M_+ = M'_+ \otimes M''_+  = (R'' \otimes R')(R' \otimes R'') \in H \otimes H$ denotes the M-matrix of $H$.
 
 In order to prove Equation~\eqref{E:general_decomposition}, it is sufficient to notice that
 \[
  J_H(W) = \sum_{\omega \in H^2(W,M;G)} J_H(W,\omega). \qedhere
 \]
\end{proof}

\begin{remark}
 Notice that, if $G$ is a finite abelian group and $H$ is a factorizable ribbon Hopf $G$-coalgebra, then
 \[
  \sum_{\alpha \in G} \lambda_\alpha(v_\alpha^{-1}) \lambda_{-\alpha}(v_{-\alpha}) = 1.
 \]
 Indeed, thanks to Proposition~\ref{P:general_decomposition}, the invariant of
 \[
  \pic{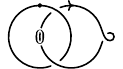}
 \]
 coincides with the invariant of
 \[
  \sum_{\alpha \in G} \pic{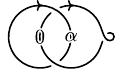} 
  = \sum_{\alpha \in G} \pic{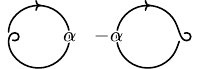}
 \]
 This means that, if $\lambda(v_\alpha^{-1})$ is an invertible scalar for a single $\alpha \in G$, then we can understand the invariant appearing on the right-hand side of Equations~\eqref{E:general_decomposition_omega} \& \eqref{E:general_decomposition} as the invariant of a \mnfld{3} equipped with a $G$-structure that assigns the degree $\alpha$ to stabilizations by $\partial \bbC P^2$ and the degree $-\alpha$ to stabilizations by $\partial \overline{\bbC P^2}$. This will be the case in Section~\ref{S:decompositions}, where we will decompose the refined and the non-refined invariants associated with the restricted quantum group $U$ in terms of \dmnsnl{3} invariants depending on spin and cohomological structures.
\end{remark}

\section{Restricted quantum \texorpdfstring{$\fsl_2$}{sl(2)}}

In this section, we recall the definition of our main examples of Hopf $\Z/2\Z$-coalgebras, the restricted quantum group of $\fsl_2$ and its ribbon extension. Let $q = e^{\frac{\pi i}{p}}$ be a root of unity of order $2p$ for an integer $p \geqs 2$. For every integer $n \geqs 0$ we recall the notation 
\begin{align*}
 \{ n \} &:= q^n - q^{-n}, &
 [n] &:= \frac{\{ n \}}{\{ 1 \}}, &
 [n]! &:= \prod_{k=1}^n [k].
\end{align*} 
The \textit{restricted quantum group\index{quantum sl(2)}\index{quantum sl(2)!restricted} $U = U_q \fsl_2$} is defined as the algebra over $\bbC$ with generators $\{ E,F,K \}$ and relations
\begin{align*}
 E^p &= 0, & F^p &= 0, &
 K^{2p} &= 1, \\
 K E K^{-1} &= q^2 E, & K F K^{-1} &= q^{-2} F, &
 [E,F] &= \frac{K - K^{-1}}{q-q^{-1}}.
\end{align*}
A Hopf algebra structure on $U$ is obtained by setting
\begin{align*}
 \Delta(E) &= E \otimes K + 1 \otimes E, & \varepsilon(E) &= 0, & S(E) &= -E K^{-1}, \\
 \Delta(F) &= F \otimes 1 + K^{-1} \otimes F, & \varepsilon(F) &= 0, & S(F) &= - K F \\
 \Delta(K) &= K \otimes K, & \varepsilon(K) &= 1, & S(K) &= K^{-1}.
\end{align*}

The restricted quantum group $U$ is not quasitriangular, because it does not admit an R-matrix. However, it admits a quasitriangular extension $\tilde{U} = \tilde{U}_q \fsl_2$, which is the algebra over $\bbC$ with generators $\{ \tilde{E},\tilde{F},\tilde{K} \}$ and relations
\begin{align*}
 \tilde{E}^p &= 0 & \tilde{F}^p &= 0, &
 \tilde{K}^{4p} &= 1, \\
 \tilde{K} \tilde{E} \tilde{K}^{-1} &= q \tilde{E}, &
 \tilde{K} \tilde{F} \tilde{K}^{-1} &= q^{-1} \tilde{F}, &
 [\tilde{E},\tilde{F}] &= \frac{\tilde{K}^2 - \tilde{K}^{-2}}{q-q^{-1}}.
\end{align*}
A Hopf algebra structure on $\tilde{U}$ is obtained by setting
\begin{align*}
 \Delta(\tilde{E}) &= \tilde{E} \otimes \tilde{K}^2 + 1 \otimes \tilde{E}, & \varepsilon(\tilde{E}) &= 0, & S(\tilde{E}) &= -\tilde{E} \tilde{K}^{-2}, \\
 \Delta(\tilde{F}) &= \tilde{F} \otimes 1 + \tilde{K}^{-2} \otimes \tilde{F}, & \varepsilon(\tilde{F}) &= 0, & S(\tilde{F}) &= - \tilde{K}^2 \tilde{F} \\
 \Delta(\tilde{K}) &= \tilde{K} \otimes \tilde{K}, & \varepsilon(\tilde{K}) &= 1, & S(\tilde{K}) &= \tilde{K}^{-1},
\end{align*}
and we identify $U$ with a Hopf subalgebra of $\tilde{U}$ by setting
\begin{align*}
 E &= \tilde{E}, &
 F &= \tilde{F}, &
 K &= \tilde{K}^2.
\end{align*}
The R-matrix $\tilde{R} \in \tilde{U} \otimes \tilde{U}$ is given by the product $\tilde{R} = \tilde{D} \Theta$, where the quasi-R-matrix $\Theta = \Theta' \otimes \Theta'' \in U \otimes U$ is given by
\begin{align}\label{E:quasi-R-matrix}
 \Theta &= \sum_{a=0}^{p-1} \frac{\{ 1 \}^a}{[a]!} q^{\frac{a(a-1)}{2}} E^a \otimes F^a,
\end{align}
while the diagonal Cartan part $\tilde{D} = \tilde{D}' \otimes \tilde{D}'' \in \tilde{U} \otimes \tilde{U}$ is given by
\begin{align}\label{E:diagonal-component}
 \tilde{D} &= \frac{1}{4p} \sum_{a,b=0}^{4p-1} t^{-ab} \tilde{K}^a \otimes \tilde{K}^b,
\end{align}
where $t := e^{\frac{\pi i}{2p}}$. Notice that it is only the diagonal Cartan part $\tilde{D}$ that cannot be defined in $U$, while the quasi-R-matrix $\Theta$ poses no problem.

Although $\tilde{R} \not\in U \otimes U$, the rest of the ribbon structure of $\tilde{U}$ is actually contained in $U$. Indeed, the pivotal element $g \in U$ can be defined as
\begin{equation}\label{E:pivotal_element}
 g = K^{p+1}.
\end{equation}
This means that the ribbon element and its inverse $v_+,v_- \in U$ are given by
\begin{align}
 v_+ &= \frac{1-i}{2\sqrt{p}} 
 \sum_{a=0}^{p-1} \sum_{b=0}^{2p-1} \frac{\{ -1 \}^a}{[a]!}
 q^{-\frac{(a+3)a}{2}} t^{(b+p+1)^2} F^a E^a K^{-a+b}, \label{E:ribbon_element} \\*
 v_- &= \frac{1+i}{2\sqrt{p}} 
 \sum_{a=0}^{p-1} \sum_{b=0}^{2p-1} \frac{\{ 1 \}^a}{[a]!} 
 q^{\frac{(a+3)a}{2}} t^{-(b+p-1)^2} F^a E^a K^{a+b}, \label{E:inverse_ribbon_element}
\end{align}
as shown in Lemma~\ref{L:ribbon}, and that the M-matrix and its inverse $M_+ = M'_+ \otimes M''_+$, $M_- = M'_- \otimes M''_- \in U \otimes U$ are given by
\begin{align}
 M_+ &= \frac{1}{2p} 
 \sum_{a,b=0}^{p-1} \sum_{c,d=0}^{2p-1} \frac{\{ 1 \}^{a+b}}{[a]![b]!} q^{\frac{a(a-1)+b(b-1)}{2} - 2b^2 - cd} K^{-b+c} F^b E^a \otimes K^{b+d} E^b F^a, \label{E:M-matrix} \\*
 M_- &= \frac{1}{2p} 
 \sum_{a,b=0}^{p-1} \sum_{c,d=0}^{2p-1} \frac{\mathclap{\{ -1 \}^{a+b}}}{\hspace*{1pt} [a]![b]! \hspace*{1pt}} q^{-\frac{a(a-1)+b(b-1)}{2} + 2b^2 + cd} E^a F^b K^{-b+c} \otimes F^a E^b K^{b+d}, \label{E:inverse_M-matrix}
\end{align}
as shown in Lemma~\ref{L:monodromy}.

A PBW basis of $U$ is given by
\[
 \left\{ E^a F^b K^c \mid 0 \leqs a,b \leqs p - 1, 0 \leqs c \leqs 2p-1 \right\},
\]
and our preferred non-zero left integral $\lambda$ of $U$ is given by
\begin{equation}\label{E:integral}
 \lambda \left( E^a F^b K^c \right) = 
 \frac{\sqrt{2p} [p-1]!}{i^{p-1} \{ 1 \}^{p-1}}
 \delta_{a,p-1} \delta_{b,p-1} \delta_{c,p-1},
\end{equation}
while our preferred non-zero two-sided cointegral $\Lambda$ of $U$ satisfying $\lambda(\Lambda) = 1$ is given by
\begin{equation}\label{E:cointegral}
 \Lambda := 
 \frac{i^{p-1} \{ 1 \}^{p-1}}{\sqrt{2p} [p-1]!}
 \sum_{a=0}^{2p-1} E^{p-1} F^{p-1} K^a.
\end{equation}
Similarly, a PBW basis of $\tilde{U}$ is given by
\[
 \left\{ E^a F^b \tilde{K}^c \mid 0 \leqs a,b \leqs p - 1, 0 \leqs c \leqs 4p-1 \right\},
\]
and our preferred non-zero left integral $\tilde{\lambda}$ of $\tilde{U}$ is given by
\begin{equation*}
 \tilde{\lambda} \left( E^a F^b \tilde{K}^c \right) = 
 \frac{\sqrt{2p} [p-1]!}{i^{p-1} \{ 1 \}^{p-1}}
 \delta_{a,p-1} \delta_{b,p-1} \delta_{c,2p-2},
\end{equation*}
while our preferred non-zero two-sided cointegral $\tilde{\Lambda}$ of $\tilde{U}$ satisfying $\tilde{\lambda}(\tilde{\Lambda}) = 1$ is given by
\begin{equation*}
 \tilde{\Lambda} := 
 \frac{i^{p-1} \{ 1 \}^{p-1}}{\sqrt{2p} [p-1]!}
 \sum_{a=0}^{4p-1} E^{p-1} F^{p-1} \tilde{K}^a.
\end{equation*}
Both $U$ and $\tilde{U}$ are thus unimodular Hopf algebras.

Now let us consider the central orthogonal idempotents
\begin{align*}
 1_0 &= \frac{1+K^p}{2}, &
 1_1 &= \frac{1-K^p}{2}.
\end{align*}
We have:
\begin{align*}
 1 &= 1_0 + 1_1, &
 1_\alpha 1_\beta &= \delta_{\alpha,\beta} 1_\alpha, \\
 \Delta(1_0) &= 1_0 \otimes 1_0 + 1_1 \otimes 1_1, &
 \Delta(1_1) &= 1_0 \otimes 1_1 + 1_1 \otimes 1_0, \\
 \varepsilon(1_0) &= 1, &
 \varepsilon(1_1) &= 0, \\
 S(1_0) &= 1_0, &
 S(1_1) &= 1_1.
\end{align*}
Therefore, $U$ is a unimodular Hopf $\Z/2\Z$-coalgebra, while $\tilde{U}$ is a ribbon one.

Recall that, if $H$ is a finite-dimensional Hopf algebra, its \textit{adjoint representation} is the vector space $H$ equipped with the adjoint left $H$-action
\[
 x \triangleright y := x_{(1)}yS(x_{(2)})
\]
for all $x,y \in H$.

A grading
\[
 U = \bigoplus_{n=-p+1}^{p-1} \Gamma_n(U)
\]
is obtained by setting $\Gamma_n(U)$ to be the linear subspace of $U$ with basis
\[
 \{ E^a F^b K^c \mid 0 \leqs a,b \leqs p - 1, 0 \leqs c \leqs 2p-1, a-b=n \},
\]
for every $-p+1 \leqs n \leqs p-1$. If $x \in \Gamma_n(U)$, then we say $x$ is homogeneous of degree $n$, and we write $\lvert x \rvert = n$.

\begin{lemma}
 The linear subspace $U$ of $\tilde{U}$ is closed under the adjoint left action of $\tilde{U}$.
\end{lemma}

\begin{proof}
 For every homogeneous $x \in U$ we have
 \[
  \tilde{K} \triangleright x = \tilde{K} x \tilde{K}^{-1} = q^{\lvert x \rvert} x. \qedhere
 \]
\end{proof}

\begin{lemma}\label{L:adjoint_action_of_D}
 For all homogeneous $x,y \in U$ we have
 \begin{align}
  (\tilde{D}' \triangleright x) \otimes \tilde{D}'' &= x \otimes K^{\lvert x \rvert}, \label{E:diagonal_adjoint_action_1} \\*
  (\tilde{D}' \triangleright x) \otimes (\tilde{D}'' \triangleright y) &= q^{2 \lvert x \rvert \lvert y \rvert} x \otimes y. \label{E:diagonal_adjoint_action_2} 
 \end{align}
\end{lemma}

\begin{proof}
 The first equality follows from
 \begin{align*}
  (\tilde{D}' \triangleright x) \otimes \tilde{D}'' &= 
  \frac{1}{4p} \sum_{a,b=0}^{4p-1} t^{-ab} \tilde{K}^a x \tilde{K}^{-a} \otimes \tilde{K}^b 
  = \frac{1}{4p} \sum_{a,b=0}^{4p-1} t^{2 a \lvert x \rvert -ab} x \otimes \tilde{K}^b \\*
  &= \sum_{b=0}^{4p-1} \left( \frac{1}{4p} \sum_{a=0}^{4p-1} t^{a(2 \lvert x \rvert -b)} \right) x \otimes \tilde{K}^b 
  = \sum_{b=0}^{4p-1} \delta_{b,2 \lvert x \rvert} x \otimes \tilde{K}^b \\*
  &= x \otimes \tilde{K}^{2 \lvert x \rvert} 
  = x \otimes K^{\lvert x \rvert}.
 \end{align*}
 Similarly, the second equality follows from
 \begin{align*}
  (\tilde{D}' \triangleright x) \otimes (\tilde{D}'' \triangleright y) &= 
  \frac{1}{4p} \sum_{a,b=0}^{4p-1} t^{-ab} \tilde{K}^a x \tilde{K}^{-a} \otimes \tilde{K}^b y \tilde{K}^{-b} \\*
  &= \frac{1}{4p} \sum_{a,b=0}^{4p-1} t^{2 a \lvert x \rvert + 2 b \lvert y \rvert -ab} x \otimes y \\*
  &= \sum_{b=0}^{4p-1} \left( \frac{1}{4p} \sum_{a=0}^{4p-1} t^{a(2 \lvert x \rvert -b)} \right) t^{2 b \lvert y \rvert} x \otimes y \\*
  &= \sum_{b=0}^{4p-1} \delta_{b,2 \lvert x \rvert} t^{2 b \lvert y \rvert} x \otimes y 
  = t^{4 \lvert x \rvert \lvert y \rvert} x \otimes y 
  = q^{2 \lvert x \rvert \lvert y \rvert} x \otimes y. \qedhere
 \end{align*}
\end{proof}

We denote by $\myuline{U} \in \mods{\tilde{U}}$ the vector space $U$ equipped with the adjoint left action of $\tilde{U}$, and with $\myuline{\tilde{U}} \in \mods{\tilde{U}}$ the adjoint representation of $\tilde{U}$.

\begin{proposition}\label{P:3-modular}
 $\myuline{U}$ is a factorizable BPH algebra in $\mods{\tilde{U}}$, with structure morphisms given, for all $x,y \in \myuline{U}$, by
 \begin{align*}
  \myuline{\mu}(x \otimes y) &= xy, &
  \myuline{\eta}(1) &= 1, \\*
  \myuline{\Delta}(x) &= x_{(1)}S(\Theta'')K^{-\lvert \Theta' \triangleright x_{(2)} \rvert} \otimes (\Theta' \triangleright x_{(2)}), &
  \myuline{\varepsilon}(x) &= \varepsilon(x), \\*
  \myuline{S}(x) &= K^{\lvert \Theta' \triangleright x \rvert} \Theta'' S(\Theta' \triangleright x), &
  \myuline{S^{-1}}(x) &= S^{-1}(\Theta' \triangleright x) K^{\lvert \Theta' \triangleright x \rvert} \Theta'', \\
  \myuline{v_+}(1) &= v_+, &
  \myuline{v_-}(1) &= v_-, \\*
  \myuline{w_+}(1) &= S(M'_+) \otimes M''_+, &
  \myuline{w_-}(1) &= S(M'_-) \otimes M''_-, \\
  \myuline{\lambda}(x) &= \lambda(x), &
  \myuline{\Lambda}(1) &= \Lambda.
 \end{align*}
\end{proposition}

\begin{proof}
 Since $\tilde{U}$ is a unimodular ribbon Hopf algebra, \cite[Proposition~7.3]{BD21} implies that $\myuline{\tilde{U}} \in \mods{\tilde{U}}$ is a BPH algebra. Structure morphisms of $\myuline{\tilde{U}}$ are given, for all $x,y \in \myuline{\tilde{U}}$, by
 \begin{align*}
  \myuline{\tilde{\mu}}(x \otimes y) &= xy, &
  \myuline{\tilde{\eta}}(1) &= 1, \\*
  \myuline{\tilde{\Delta}}(x) &= x_{(1)}S(\tilde{D}''\Theta'') \otimes ((\tilde{D}'\Theta') \triangleright x_{(2)}), &
  \myuline{\tilde{\varepsilon}}(x) &= \varepsilon(x), \\*
  \myuline{\tilde{S}}(x) &= \tilde{D}''\Theta'' S((\tilde{D}'\Theta') \triangleright x), &
  \myuline{\tilde{S}^{-1}}(x) &= S^{-1}((\tilde{D}'\Theta') \triangleright x) \tilde{D}''\Theta'', \\
  \myuline{\tilde{v}_+}(1) &= v_+, &
  \myuline{\tilde{v}_-}(1) &= v_-, \\*
  \myuline{\tilde{w}_+}(1) &= S(M'_+) \otimes M''_+, &
  \myuline{\tilde{w}_-}(1) &= S(M'_-) \otimes M''_-, \\
  \myuline{\tilde{\lambda}}(x) &= \tilde{\lambda}(x), &
  \myuline{\tilde{\Lambda}}(1) &= \tilde{\Lambda}.
 \end{align*}
 Since $U$ is a Hopf subalgebra of $\tilde{U}$, the unit $\myuline{\tilde{\eta}} : \bbC \to \myuline{\tilde{U}}$ factors through a unit $\myuline{\eta} : \bbC \to \myuline{U}$, while the product $\myuline{\tilde{\mu}} : \myuline{\tilde{U}} \otimes \myuline{\tilde{U}} \to \myuline{\tilde{U}}$ and the counit $\myuline{\tilde{\varepsilon}} : \myuline{\tilde{U}} \to \bbC$ restrict to a product $\myuline{\mu} : \myuline{U} \otimes \myuline{U} \to \myuline{U}$ and a counit $\myuline{\varepsilon} : \myuline{U} \to \bbC$. Equations~\eqref{E:ribbon_element} \& \eqref{E:inverse_ribbon_element} imply that the ribbon element $\myuline{\tilde{v}_+} : \bbC \to \myuline{\tilde{U}}$ and its inverse $\myuline{\tilde{v}_-} : \bbC \to \myuline{\tilde{U}}$ factor through a ribbon element $\myuline{v_+} : \bbC \to \myuline{U}$ with inverse $\myuline{v_-} : \bbC \to \myuline{U}$, and Equation~\eqref{E:integral} implies that the integral $\myuline{\tilde{\lambda}} : \myuline{U} \to \bbC$ restricts to an integral $\myuline{\lambda} : \myuline{U} \to \bbC$. However, it should be noted that the cointegral $\myuline{\tilde{\Lambda}} : \bbC \to \myuline{\tilde{U}}$ is different from the cointegral $\myuline{\Lambda} : \bbC \to \myuline{U}$, as witnessed by Equation~\eqref{E:cointegral}. The braiding $c_{\myuline{\tilde{U}},\myuline{\tilde{U}}} : \myuline{\tilde{U}} \otimes \myuline{\tilde{U}} \to \myuline{\tilde{U}} \otimes \myuline{\tilde{U}}$ restricts to a braiding $c_{\myuline{U},\myuline{U}} : \myuline{U} \otimes \myuline{U} \to \myuline{U} \otimes \myuline{U}$, because $\mods{\tilde{U}}$ is a ribbon category, although a direct proof follows from Equation~\eqref{E:diagonal_adjoint_action_2}. Therefore, we need to show that the coproduct $\myuline{\tilde{\Delta}} : \myuline{\tilde{U}} \to \myuline{\tilde{U}} \otimes \myuline{\tilde{U}}$, the antipode $\myuline{\tilde{S}} : \myuline{\tilde{U}} \to \myuline{\tilde{U}}$, and its inverse $\myuline{\tilde{S}^{-1}} : \myuline{\tilde{U}} \to \myuline{\tilde{U}}$ restrict to a coproduct $\myuline{\Delta} : \myuline{U} \to \myuline{U} \otimes \myuline{U}$ and an antipode $\myuline{S} : \myuline{U} \to \myuline{U}$ with inverse $\myuline{S^{-1}} : \myuline{U} \to \myuline{U}$. Notice that
 \begin{align*}
  \myuline{\tilde{\Delta}}(x) &= x_{(1)}S(\tilde{D}''\Theta'') \otimes (\tilde{D}' \triangleright (\Theta' \triangleright x_{(2)})), \\*
  \myuline{\tilde{S}}(x) &= \tilde{D}''\Theta'' S(\tilde{D}' \triangleright (\Theta' \triangleright x)), \\*
  \myuline{\tilde{S}^{-1}}(x) &= S^{-1}(\tilde{D}' \triangleright (\Theta' \triangleright x)) \tilde{D}''\Theta''.
 \end{align*}
 Then the claim follows directly from Equation~\eqref{E:diagonal_adjoint_action_1}.
 
 Next, we need to check that these structure morphisms satisfy the defining conditions of \cite[Definitions~5.1--6.4]{BD21}. For what concerns \cite[Definition~5.1]{BD21}, Equation~$(i)$ is clearly satisfied because $U$ is an associative unital algebra, while Equations~$(ii)$--$(iv)$ are satisfied because $\myuline{\Delta}$, $\myuline{\varepsilon}$, $\myuline{S}$, and $\myuline{S^{-1}}$ are restrictions of $\myuline{\tilde{\Delta}}$, $\myuline{\tilde{\varepsilon}}$, $\myuline{\tilde{S}}$, and $\myuline{\tilde{S}^{-1}}$ respectively, which also satisfy Equations~$(ii)$--$(iv)$. For what concerns \cite[Definition~6.1]{BD21}, Equations~$(i)$--$(iii)$ are satisfied because $\myuline{\tilde{v}_+}$, $\myuline{\tilde{v}_-}$, $\myuline{\tilde{w}_+}$, $\myuline{\tilde{w}_-}$, and $\myuline{\tilde{\eta}}$ factor through $\myuline{v_+}$, $\myuline{v_-}$, $\myuline{w_+}$, $\myuline{w_-}$, and $\myuline{\eta}$, while $\myuline{\mu}$, $\myuline{\Delta}$, $\myuline{\varepsilon}$, $\myuline{S}$, $\myuline{S^{-1}}$, and $c_{\myuline{U},\myuline{U}}$ are restrictions of $\myuline{\tilde{\mu}}$, $\myuline{\tilde{\Delta}}$, $\myuline{\tilde{\varepsilon}}$, $\myuline{\tilde{S}}$, $\myuline{\tilde{S}^{-1}}$, and $c_{\myuline{\tilde{U}},\myuline{\tilde{U}}}$ respectively, which also satisfy Equations~$(i)$--$(iii)$. For what concerns \cite[Definition~6.2]{BD21}, Equation~$(i)$ is satisfied because $\myuline{\lambda}$, $\myuline{\Delta}$, and $\myuline{S}$ are restrictions of $\myuline{\tilde{\lambda}}$, $\myuline{\tilde{\Delta}}$, and $\myuline{\tilde{S}}$ respectively, which also satisfy Equation~$(i)$. Therefore, we need to check by hand Equations~$(ii)$ \& $(iii)$, which are the only ones involving $\Lambda$. The first part of Equation~$(ii)$ follows from the fact that $\Lambda$ is a left cointegral of $U$, while the second part is established by computing
 \begin{align*}
  \myuline{S}(\myuline{\Lambda}) = K^{\lvert \Theta' \triangleright \Lambda \rvert} \Theta'' S(\Theta' \triangleright \Lambda) = S(\Lambda) = \Lambda = \myuline{\Lambda}.
 \end{align*}
 Equation~$(iii)$ follows from our choice of normalizations for $\lambda$ and $\Lambda$. 
 
 Finally, for what concerns \cite[Definition~6.4]{BD21}, we compute
 \begin{align*}
  &(\id_{\myuline{U}} \otimes \myuline{\lambda}) \circ \myuline{w_+} 
  = \lambda(M''_+) S(M'_+) \\*
  &\hspace*{\parindent} = \frac{1}{2p} 
 \sum_{a,b=0}^{p-1} \sum_{c,d=0}^{2p-1} \frac{\{ 1 \}^{a+b}}{[a]![b]!} q^{\frac{a(a-1)+b(b-1)}{2} - 2b^2 - cd} \lambda(K^{b+d} E^b F^a) S(K^{-b+c} F^b E^a) \\*
  &\hspace*{\parindent} = \frac{1}{\sqrt{2p}} \sum_{c=0}^{2p-1} \frac{\{ 1 \}^{p-1}}{i^{p-1} [p-1]!} q^{(p-1)(p-2) - 2(p-1)^2} S(K^{c-p+1} F^{p-1} E^{p-1}) \\*
  &\hspace*{\parindent} = \frac{q^{-p(p-1)}}{i^{p-1} \sqrt{2p}} \frac{\{ 1 \}^{p-1}}{[p-1]!} \sum_{c=0}^{2p-1} E^{p-1} F^{p-1} K^{-c+p-1} \\*
  &\hspace*{\parindent} = \frac{(-1)^{p-1}}{i^{p-1} \sqrt{2p}} \frac{\{ 1 \}^{p-1}}{[p-1]!} \sum_{a=0}^{2p-1} E^{p-1} F^{p-1} K^a 
  = \Lambda = \myuline{\Lambda}. \qedhere
 \end{align*}
\end{proof}

As a direct consequence of \cite[Theorem~1.6]{BD21} and of Proposition~\ref{P:3-modular}, we immediately obtain the following result.

\begin{corollary}\label{C:U_factorizable}
 There exists a unique $3$-di\-men\-sion\-al braided TFT\index{TQFT} 
 \[
  J_3^\sigma : \FRCob \to \mods{\tilde{U}}
 \]
 sending $1 \in \FRCob$ (the punctured torus) to $\myuline{U} \in \mods{\tilde{U}}$.
\end{corollary}

\section{Refined invariant for restricted quantum \texorpdfstring{$\fsl_2$}{sl(2)}}

In this section, we prove that the refined invariant associated with the ribbon extension $\tilde{U}$ of $U$ can be actually computed entirely inside $U$.

\begin{proposition}\label{P:restricted}
 If $D$ is a $\Z/2\Z$-Kirby diagram, and if 
 \[
  x_1 \otimes \ldots \otimes x_k \in \tilde{U}_{\alpha_1} \otimes \ldots \otimes \tilde{U}_{\alpha_k}
 \]
 is a bead presentation of $D$, as defined in Equation~\eqref{E:bead_presentation}, then
 \[
  x_1 \otimes \ldots \otimes x_k \in U_{\alpha_1} \otimes \ldots \otimes U_{\alpha_k}.
 \]
\end{proposition}

\begin{proof}
 The claim follows from \cite[Theorem~4.7.5]{BP11} and Proposition~\ref{P:3-modular}. Indeed, every Kirby diagram can be realized as the composition of tensor products of generating morphisms appearing in the definition of the Kirby tangle presentation functor $K : \Algf \to \KTan$ of \cite[Section~6.3]{BD21}. Then, it is sufficient to check that, for each of these generating morphisms, the algorithm defining $J_4 : \KTan \to \mods{\tilde{U}}$ determines elements of $U$. This corresponds to the computations of Proposition~\ref{P:3-modular}, and these formulas can be established like in the proofs of \cite[Lemmas~8.1 \& 8.3]{BD21}. For instance, the coproduct gives
 \[
  \pic{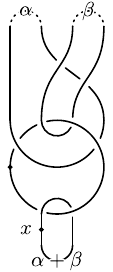} = \pic{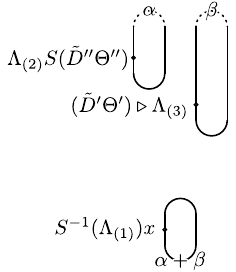}
 \]
 Then the claim follows from the chain of equalities
 \begin{align*}
  &\lambda(S^{-1}(\Lambda_{(1)}) x 1_{\alpha+\beta}) \Lambda_{(2)}S(\tilde{D}'' \Theta'') 1_\alpha \otimes ((\tilde{D}' \Theta') \triangleright \Lambda_{(3)}) 1_\beta \\*
  &\hspace*{\parindent} = \lambda(x 1_{\alpha+\beta} S(\Lambda_{(1)})) \Lambda_{(2)}S(\tilde{D}'' \Theta'') 1_\alpha \otimes ((\tilde{D}' \Theta') \triangleright \Lambda_{(3)}) 1_\beta \\*
  &\hspace*{\parindent} = (x 1_{\alpha+\beta})_{(1)} S(\tilde{D}'' \Theta'') 1_\alpha \otimes ((\tilde{D}' \Theta') \triangleright (x 1_{\alpha+\beta})_{(2)}) 1_\beta \\*
  &\hspace*{\parindent} = x_{(1)} S(\Theta'') K^{-\lvert \Theta' \triangleright x_{(2)} \rvert} 1_\alpha \otimes (\Theta' \triangleright x_{(2)}) 1_\beta.
 \end{align*}
 where the first equality follows from \cite[Theorem~10.5.4.(e)]{R12}, while the second one follows from \cite[Theorem~10.2.2.(c)]{R12}. Similarly, the antipode gives
 \[
  \pic{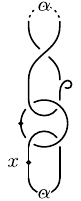} = \pic{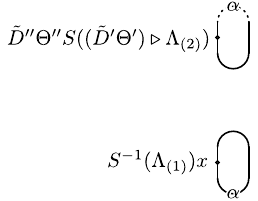}
 \]
 Then the claim follows from the chain of equalities
 \begin{align*}
  &\lambda(S^{-1}(\Lambda_{(1)}) x 1_\alpha) \tilde{D}'' \Theta'' S((\tilde{D}' \Theta') \triangleright \Lambda_{(2)}) 1_\alpha \\*
  &\hspace*{\parindent} = \lambda(x 1_\alpha S(\Lambda_{(1)})) \tilde{D}'' \Theta'' S((\tilde{D}' \Theta') \triangleright \Lambda_{(2)}) 1_\alpha \\*
  &\hspace*{\parindent} = \tilde{D}'' \Theta'' S((\tilde{D}' \Theta') \triangleright (x 1_\alpha)) 1_\alpha \\*
  &\hspace*{\parindent} = K^{\lvert \Theta' \triangleright x \rvert} \Theta'' S(\Theta' \triangleright x) 1_\alpha. \qedhere
 \end{align*}
\end{proof}

Because of Proposition~\ref{P:restricted}, we will use the notation $J_U$ for the invariant given by Theorem~\ref{T:main} with $H = \tilde{U}$.

\section{Relation between invariants for restricted and small quantum \texorpdfstring{$\fsl_2$}{sl(2)}}

In this section, we prove that the non-refined invariant associated with the small quantum group of $\fsl_2$ coincides with the refined invariant associated with the restricted quantum group of $\fsl_2$ for the trivial cohomology class. In order to do this, let us set
\[
 p' := \frac{p}{\gcd(p,2)}.
\]
The \textit{small quantum group\index{quantum sl(2)}\index{quantum sl(2)!small} $\bar{U} = \bar{U}_q \fsl_2$} is defined as the algebra over $\bbC$ with generators $\{ \bar{E},\bar{F},\bar{K} \}$ and relations
\begin{align*}
 \bar{E}^p &= 0, & \bar{F}^p &= 0, &
 \bar{K}^p &= 1, \\
 \bar{K} \bar{E} \bar{K}^{-1} &= q^2 \bar{E}, & \bar{K} \bar{F} \bar{K}^{-1} &= q^{-2} \bar{F}, &
 [\bar{E},\bar{F}] &= \frac{\bar{K} - \bar{K}^{-1}}{q-q^{-1}}.
\end{align*}
We identify $\bar{U}$ with a Hopf subalgebra of $U$ by setting
\begin{align*}
 \bar{E} &= E 1_0, &
 \bar{F} &= F 1_0, &
 \bar{K} &= K 1_0.
\end{align*}
As explained in \cite[Section~9.2]{BD21}, $\bar{U}$ is a unimodular ribbon Hopf algebra, and it is factorizable if and only if $p \not\equiv 0 \pmod 2$. Notice that the R-matrix $\bar{R} \in \bar{U} \otimes \bar{U}$ is given by the product $\bar{R} = \bar{D} \bar{\Theta}$, where the quasi-R-matrix $\bar{\Theta} = \bar{\Theta}' \otimes \bar{\Theta}'' \in \bar{U} \otimes \bar{U}$ is given by
\begin{align}\label{E:small-quasi-R-matrix}
 \bar{\Theta} = \Theta' 1_0 \otimes \Theta'' 1_0,
\end{align}
while the diagonal Cartan part $\bar{D} = \bar{D}' \otimes \bar{D}'' \in \bar{U} \otimes \bar{U}$ is given by
\begin{align}\label{E:small-diagonal-component}
 \bar{D} &= \frac{1}{p} \sum_{a,b=0}^{p-1} q^{-2ab} \bar{K}^a \otimes \bar{K}^b.
\end{align}
Let us now restrict our attention to the case $p \equiv 0 \pmod 2$. The ribbon element and its inverse $\bar{v}_+,\bar{v}_- \in \bar{U}$ are given for $p \equiv 2 \pmod 2$ by
\begin{align}
 \bar{v}_+ 
 &= \frac{i^{\frac{p'-1}{2}}}{\sqrt{p'}}
 \sum_{a=0}^{p-1} \sum_{b=0}^{p'-1} \frac{\{ -1 \}^a}{[a]!}
 q^{-\frac{(a+3)a}{2} + \frac{(p'+1)^3}{2} (2b-1)^2} \bar{F}^a \bar{E}^a \bar{K}^{-a-2b}, \label{E:ribbon_small_1} \\*
 \bar{v}_- 
 &= \frac{i^{-\frac{p'-1}{2}}}{\sqrt{p'}} \sum_{a=0}^{p-1} \sum_{b=0}^{p'-1} \frac{\{ 1 \}^a}{[a]!}
 q^{\frac{(a+3)a}{2} - \frac{(p'+1)^3}{2} (2b-1)^2} \bar{F}^a \bar{E}^a \bar{K}^{a+2b}, \label{E:inverse_ribbon_small_1}
\end{align}
and for $p \equiv 0 \pmod 4$ by
\begin{align}
 \bar{v}_+ 
 &= \frac{1-i}{\sqrt{p}} 
 \sum_{a=0}^{p-1} \sum_{b=0}^{p'-1} \frac{\{ -1 \}^a}{[a]!} q^{-\frac{(a+3)a}{2} + 2b^2} \bar{F}^a \bar{E}^a \bar{K}^{-a-2b-1}, \label{E:ribbon_small_2} \\*
 \bar{v}_- &= 
 \frac{1+i}{\sqrt{p}} 
 \sum_{a=0}^{p-1} \sum_{b=0}^{p'-1} \frac{\{ 1 \}^a}{[a]!} q^{\frac{(a+3)a}{2} - 2b^2} \bar{F}^a \bar{E}^a \bar{K}^{a+2b+1}, \label{E:inverse_ribbon_small_2}
\end{align}
as shown in \cite[Lemma~B.2, Equations~(B.10)--(B.13)]{BD21}. Similarly, the M-matrix and its inverse $\bar{M}_+ = \bar{M}'_+ \otimes \bar{M}''_+$, $\bar{M}_- = \bar{M}'_- \otimes \bar{M}''_- \in \bar{U} \otimes \bar{U}$ are given by
\begin{align}
  \bar{M} &=
  \frac{1}{p'} \sum_{a,b=0}^{p-1} \sum_{c,d=0}^{p'-1} \frac{\{ 1 \}^{a+b}}{[a]![b]!} \nonumber \\*
  &\hspace*{\parindent} q^{\frac{a(a-1)+b(b-1)}{2} - 2b^2 - 4cd} \bar{K}^{-b+2c} \bar{F}^b \bar{E}^a \otimes \bar{K}^{b+2d} \bar{E}^b \bar{F}^a, \label{E:M_small_1} \\
  \bar{M}^{-1} &=
  \frac{1}{p'} \sum_{a,b=0}^{p-1} \sum_{c,d=0}^{p'-1} \frac{\{ -1 \}^{a+b}}{[a]![b]!} \nonumber \\*
  &\hspace*{\parindent} q^{\frac{a(a-1)+b(b-1)}{2} + 2b^2 + 4cd} \bar{E}^a \bar{F}^b \bar{K}^{-b+2c} \otimes \bar{F}^a \bar{E}^b \bar{K}^{b+2d}. \label{E:M_inverse_small_1}
 \end{align}
as shown in \cite[Lemma~B.3, Equations~(B.31)--(B.32)]{BD21}. We slightly change our conventions with respect to those of \cite[Section~9.2]{BD21}, and fix our preferred non-zero left integral $\bar{\lambda}$ of $\bar{U}_q \fsl_2$ to be
\begin{equation}\label{E:small_integral}
 \bar{\lambda} \left( \bar{E}^a \bar{F}^b \bar{K}^c \right) = 
 \frac{\sqrt{p'} [p-1]!}{i^{p-1} \{ 1 \}^{p-1}}
 \delta_{a,p-1} \delta_{b,p-1} \delta_{c,p-1},
\end{equation}
while our preferred non-zero two-sided cointegral $\bar{\Lambda}$ of $\bar{U}_q \fsl_2$ satisfying $\lambda(\Lambda) = 1$ is
\begin{equation}\label{E:small_cointegral}
 \bar{\Lambda} := 
 \frac{i^{p-1} \{ 1 \}^{p-1}}{\sqrt{p'} [p-1]!}
 \sum_{a=0}^{p-1} \bar{E}^{p-1} \bar{F}^{p-1} \bar{K}^a.
\end{equation}
See Appendix~\ref{A:rescaling} for an explanation of the effect of this change of conventions, and for why the current ones are more natural than the ones of \cite{BD21}.

\begin{proposition}\label{P:restricted_vs_small}
 For the zero cohomology class, the refined invariant $J_U$ satisfies
 \[
  J_U(W,0) = J_{\bar{U}} (W),
 \]
 where $J_{\bar{U}}$ denotes the invariant associated with the small quantum group $\bar{U}$.
\end{proposition}

\begin{proof}
 Since $\bar{U}$ is a unimodular ribbon Hopf algebra, \cite[Proposition~7.3]{BD21} implies that $\myuline{\bar{U}} \in \mods{\bar{U}}$ is a BPH algebra. Structure morphisms of $\myuline{\bar{U}}$ are given, for all $x,y \in \myuline{\bar{U}}$, by
 \begin{align*}
  \myuline{\bar{\mu}}(x \otimes y) &= xy, &
  \myuline{\bar{\eta}}(1) &= 1, \\*
  \myuline{\bar{\Delta}}(x) &= x_{(1)}S(\bar{D}''\bar{\Theta}'') \otimes ((\bar{D}'\bar{\Theta}') \triangleright x_{(2)}), &
  \myuline{\bar{\varepsilon}}(x) &= \varepsilon(x), \\*
  \myuline{\bar{S}}(x) &= \bar{D}''\bar{\Theta}'' S((\bar{D}'\bar{\Theta}') \triangleright x), &
  \myuline{\bar{S}^{-1}}(x) &= S^{-1}((\bar{D}'\bar{\Theta}') \triangleright x) \bar{D}''\bar{\Theta}'', \\
  \myuline{\bar{v}_+}(1) &= \bar{v}_+, &
  \myuline{\bar{v}_-}(1) &= \bar{v}_-, \\*
  \myuline{\bar{w}_+}(1) &= S(\bar{M}'_+) \otimes \bar{M}''_+, &
  \myuline{\bar{w}_-}(1) &= S(\bar{M}'_-) \otimes \bar{M}''_-, \\
  \myuline{\bar{\lambda}}(x) &= \bar{\lambda}(x), &
  \myuline{\bar{\Lambda}}(1) &= \bar{\Lambda}.
 \end{align*}

 First of all, we claim that
 \begin{align*}
  \myuline{\bar{\Delta}}(x 1_0) &= \myuline{\Delta}(x)(1_0 \otimes 1_0), &
  \myuline{\bar{S}}(x 1_0) &= \myuline{S}(x)1_0
 \end{align*}
 for every $x \in U$. Indeed, just like for the restricted quantum group, a grading
 \[
  \bar{U} = \bigoplus_{n=-p+1}^{p-1} \Gamma_n(\bar{U})
 \]
 is obtained by setting $\Gamma_n(\bar{U})$ to be the linear subspace of $\bar{U}$ with basis
 \[
  \{ \bar{E}^a \bar{F}^b \bar{K}^c \mid 0 \leqs a,b,c \leqs p-1, a-b=n \},
 \]
 for every $-p+1 \leqs n \leqs p-1$. As usual, if $x \in \Gamma_n(\bar{U})$, then we say $x$ is homogeneous of degree $n$, and we write $\lvert x \rvert = n$. Then, for all homogeneous $x,y \in \bar{U}$, we have
 \begin{align*}
  (\bar{D}' \triangleright x) \otimes \bar{D}'' &= x \otimes \bar{K}^{\lvert x \rvert}, \\*
  (\bar{D}' \triangleright x) \otimes (\bar{D}'' \triangleright y) &= q^{2 \lvert x \rvert \lvert y \rvert} x \otimes y. 
 \end{align*}
 This can be shown exactly like Lemma~\ref{L:adjoint_action_of_D}, and gives a proof of the claim.

 Next, we we claim that
 \begin{align*}
  \bar{v}_+ &= v_+ 1_0, &
  \bar{w}_+ &= w_+ (1_0 \otimes 1_0).
 \end{align*}
 Indeed, on the one hand, the equality for the ribbon element is obtained by comparing Equations~\eqref{E:ribbon_small_1} \& \eqref{E:ribbon_small_2} with Equations~\eqref{E:graded_ribbon_0_1} \& \eqref{E:graded_ribbon_0_2}. For $p \equiv 2 \pmod 4$ we have
 \begin{align*}
  i &= q^{p'}, &
  q^{\frac{(p'+1)^3}{2}} &= q^{-\frac{p'^2-1}{2}}, &
  \frac{1-i}{\sqrt{2}} = t^{-p'}.
 \end{align*}
 Therefore, the equality follows from
 \begin{align*}
  i^{\frac{p'-1}{2}} q^{\frac{(p'+1)^3}{2} (2b-1)^2}
  &= q^{\frac{p'(p'-1)}{2} - \frac{p'^2-1}{2} (2b-1)^2} \\*
  &= t^{p'^2-p' - (p'^2-1)(2b-1)^2} \\*
  &= t^{(1-(2b-1)^2)p'^2 - p' + (2b-1)^2} \\*
  &= t^{-4b(b-1)p'^2} t^{-p'} t^{(2b-1)^2} \\*
  &= \frac{1-i}{\sqrt{2}} t^{(2b-1)^2}.
 \end{align*}
 For $p \equiv 0 \pmod 4$, the equality is clear. On the other hand, the equality for the copairing is obtained by computing
 \begin{align*}
  w_+ 
  &= \frac{1}{p'} 
  \sum_{a,b=0}^{p-1} \sum_{c,d=0}^{p'-1} \frac{\{ 1 \}^{a+b}}{[a]![b]!} q^{\frac{a(a-1)+b(b-1)}{2} - 2b^2 - 4cd} S(K^{-b+2c} F^b E^a) \otimes K^{b+2d} E^b F^a \\*
  &= \frac{1}{p'} 
  \sum_{a,b=0}^{p-1} \sum_{c,d=0}^{p'-1} \frac{\{ -1 \}^{a+b}}{[a]![b]!} q^{\frac{a(a-1)+b(b-1)}{2} - (a-b)(a-b-1) - 2b^2 - 4cd} \\*
  &\hspace*{\parindent} E^a F^b K^{-a+2b-2c} \otimes K^{b+2d} E^b F^a \\*
  &= \frac{1}{p'} 
  \sum_{a,b=0}^{p-1} \sum_{c,d=0}^{p'-1} \frac{\{ -1 \}^{a+b}}{[a]![b]!} q^{\frac{a(a-1)+b(b-1)}{2} - (a-b)(a-b-1) - 2(a-b)(b+2d) - 2b^2 - 4cd} \\*
  &\hspace*{\parindent} E^a F^b K^{-a+2b-2c} \otimes E^b F^a K^{b+2d} \\*
  &= \frac{1}{p'} 
  \sum_{a,b=0}^{p-1} \sum_{c,d=0}^{p'-1} \frac{\{ -1 \}^{a+b}}{[a]![b]!} q^{\frac{a(a-1)+b(b-1)}{2} - (a-b)(a-b-1) - 2(a-b)(b+2d) - 2b^2 + 4(a-b+c)d} \\*
  &\hspace*{\parindent} E^a F^b K^{a+2c} \otimes E^b F^a K^{b+2d} \\*
  &= \frac{1}{p'} 
  \sum_{a,b=0}^{p-1} \sum_{c,d=0}^{p'-1} \frac{\{ -1 \}^{a+b}}{[a]![b]!} q^{-\frac{a(a-1)+b(b-1)}{2} - 2b + 4cd} E^a F^b K^{a+2c} \otimes E^b F^a K^{b+2d},
 \end{align*}
 and comparing with Equation~\eqref{E:graded_copairing_0_0}.

 Finally, we have
 \begin{align*}
  \bar{\lambda}(x 1_0) &= \lambda(x 1_0), &
  \bar{\Lambda} &= \Lambda 1_0
 \end{align*}
 for every $x \in U$. Indeed, we have
 \begin{align*}
  \bar{\lambda}(E^a F^b K^c 1_0) &= \frac{\sqrt{p'} [p-1]!}{i^{p-1} \{ 1 \}^{p-1}}
 \delta_{a,p-1} \delta_{b,p-1} \left( \delta_{c,p-1} + \delta_{c,2p-1} \right), \\*
  \lambda(E^a F^b K^c 1_0) &= \frac{\sqrt{2p} [p-1]!}{i^{p-1} \{ 1 \}^{p-1}}
 \delta_{a,p-1} \delta_{b,p-1} \left( \frac{\delta_{c,p-1} + \delta_{c,2p-1}}{2} \right)
 \end{align*}
 for all integers $0 \leqs a,b \leqs p-1$ and $0 \leqs c \leqs 2p-1$.
\end{proof}

\section{Spin and cohomological decomposition formulas}\label{S:decompositions}

In this section, we derive a decomposition formula for both the refined and the non-refined invariants associated with the restricted quantum group $U$ in terms of refined invariants of \dmnsnl{3} boundaries equipped with additional structures. Recall that a \textit{spin structure}\index{spin structure} on a connected \dmnsnl{n} \hndlbd{k}
\[
 D^n = X_0 \subset X_1 \subset X_2 \subset \ldots \subset X_k = X
\]
can be defined as the fiber-homotopy class of a trivialization of the tangent bundle $T X_2$ of the \hndlbd{2} $X_2$. A spin structure on $X$ exists if and only if the second Stiefel--Whitney class $w_2(X) \in H^2(X;\Z/2\Z)$ vanishes. By definition, $w_2(X)$ is the cohomology class of a cocycle $c(\tau)$ that is constructed as follows: first, we pick an arbitrary trivialization $\tau$ of the tangent bundle $TX_1$ of the \hndlbd{1} $X_1$, which exists because $X$ is oriented; next, for every $2$-han\-dle in $X_2$, we use the corresponding attaching map in order to compare $\tau$ with the unique trivialization of the tangent bundle of $D^2 \times D^{n-2}$ up to fiber-homotopy; this allows us to associate with every $2$-han\-dle in $X_2$ the obstruction to extending $\tau$, which is an element in $\pi_1(\SO(n)) \cong \Z/2\Z$ (for $n > 2$). Therefore, $X$ admits a spin structure if and only if $c(\tau)$ is a coboundary for some choice of $\tau$, in which case it is a coboundary for all possible choices of $\tau$, see \cite[Section~5.6]{GS99} for more details. If $w_2(X) = 0$, then the set $\calS(X)$ of spin structures on $X$ is affinely isomorphic to the vector space 
\[
 \calH(X) := H^1(X;\Z/2\Z),
\]
since every pair $s,s' \in \calS(X)$ of spin structures on $X$ determines a \textit{difference class} $\Delta(s,s') \in \calH(X)$. Indeed, up to fiber-homotopy, we can suppose that $s$ and $s'$ coincide on the \hndlbd{0} $D^n = X_0$, and we obtain a difference cochain $d(s,s')$ inside $C^1(X;\pi_1(\SO(n))) \cong C^1(X;\Z/2\Z)$ by comparing the two restrictions to the \hndlbd{1} $X_1$. This is a cocycle because both $s$ and $s'$ extend to $X_2$, and a different choice for the fiber-homotopy ensuring that $s$ and $s'$ coincide on $X_0$ only affects $d(s,s')$ by a coboundary, see \cite[Section~5.6]{GS99}.

If $Y$ is a $(n-1)$-manifold with a spin structure $s \in \calS(Y)$, then a relative \dmnsnl{n} \hndlbd{k}
\[
 Y \times I = X_0 \subset X_1 \subset X_2 \subset \ldots \subset X_k = X
\]
admits a spin structure extending $s$ if and only if the second relative Stiefel--Whitney class $w_2(X,s) \in H^2(X,Y;\Z/2\Z)$ vanishes. The definition of $w_2(X,s)$ is analogous to that of $w_2(X)$, although this time the trivialization $\tau$ of the tangent bundle $TX_1$ of the \hndlbd{1} $X_1$ is required to extend $s$.

Every \mnfld{3} admits a spin structure, because every \mnfld{3} is parallelizable, see for instance \cite{BL18}. Notice however that there exist \mnflds{4} that do not admit any spin structure.

Let $W$ be a connected \dmnsnl{4} \hndlbd{2} with boundary $\partial W = M$. For every relative cohomology class $\omega \in H^2(W,M;\Z/2\Z)$, we set
\begin{align*}
 \calS(M,\omega) &:= \{ s \in \calS(M) \mid w_2(W,s) = \omega \} \\*
 \calH(M,\omega) &:= \{ \varphi \in \calH(M) \mid \delta^*(\varphi) = \omega \}
\end{align*}
where $\delta^* : \calH(M) \to H^2(W,M;\Z/2\Z)$ denotes the coboundary homomorphism coming from the long exact sequence of the pair $(W,M)$ in cohomology with $\Z/2\Z$-coefficients.

Let $W$ be a geometrically simple connected \dmnsnl{4} \hndlbd{2} with boundary $\partial W = M$, let $L = L_1 \cup \ldots \cup L_n$ denote a Kirby diagram for $W$ featuring only $2$-han\-dles, which yields a surgery presentation of $M$. Then, every cohomology class in $\omega \in H^2(W,M;\Z/2\Z)$ can be identified with (the indicator function of) a sublink of $L$. Indeed, if $m_i \in H_2(W,M)$ denotes the relative homology class of disc providing a meridian for the tubular neighborhood of a component $L_i \subset L$, then let us set
\[
 \omega_i := \langle \omega,m_i \rangle \in \Z/2\Z.
\]
A sublink $\omega \in H^2(W,M;\Z/2\Z)$ is said to be \textit{characteristic} if it satisfies the equation
\[
 \lk (L_i,\omega) := \sum_{j=1}^n \omega_j \lk(L_i,L_j) \equiv \lk(L_i,L_i) \pmod 2
\]
for every $1 \leqs i \leqs n$, while it is said to be \textit{even} if it satisfies the equation
\[
 \lk (L_i,\omega) := \sum_{j=1}^n \omega_j \lk(L_i,L_j) \equiv 0 \pmod 2
\]
for every $1 \leqs i \leqs n$. On the one hand, the map
\begin{align*}
 \calS(M) &\to H^2(W,M;\Z/2\Z) \\*
 s &\mapsto w_2(W,s) 
\end{align*}
defines a bijection between $\calS(M)$ and the set of characteristic sublinks of $L$, as explained in \cite[Section~5.7.11]{GS99}. In other words, for every characteristic sublink $\omega \in H^2(W,M;\Z/2\Z)$ there exists a unique spin structure $s \in \calS(M)$ such that $w_2(W,s) = \omega$. As explained above, the characteristic class $w_2(W,s)$ measures the obstruction to extending $s$ from $M$ to $W$. In particular, we have that $W$ is a spin \mnfld{4} if and only if the empty sublink $0 \in H^2(W,M;\Z/2\Z)$ is a characteristic sublink. On the other hand, the coboundary homomorphism
\begin{align*}
 \calH(M) &\to H^2(W,M;\Z/2\Z) \\*
 \varphi &\mapsto \delta^*(\varphi)
\end{align*}
coming from the long exact sequence of the pair $(W,M)$ in cohomology with $\Z/2\Z$-coefficients defines a bijection between $\calH(M)$ and the set of even sublinks of $L$. Notice that, when $W$ is not geometrically simply connected, then neither of these maps is injective in general. For instance, if $W = S^1 \times D^3$ and $M = S^1 \times S^2$, then $H^2(W,M;\Z/2\Z) = 0$ while $H^1(M;\Z/2\Z) \cong \Z/2\Z$.

\begin{theorem}\label{T:decomposition}
 Let $W$ be a connected \dmnsnl{4} \hndlbd{2} with boundary $\partial W = M$, let $D$ be a Kirby diagram for $W$, let $L$ be the framed link obtained from $D$ by trading $1$-han\-dles for $2$-han\-dles, let $\sigma$ be its signature, and let $E$ be the \dmnsnl{4} \hndlbd{2} represented by $L$. 
 \begin{enumerate}
  \item If $p \equiv 0 \pmod 4$, then for every spin structure $s \in \calS(M)$ the scalar
   \begin{equation}
    J_U(M,s) := \lambda(v_+ 1_1)^\sigma J_U(E,w_2(E,s))
   \end{equation}
   is a topological invariant of the pair $(M,s)$, for every relative cohomology class $\omega \in H^2(W,M;\Z/2\Z)$ we have
   \begin{align}
    J_U(W,\omega) &= \lambda(v_- 1_1)^\sigma \sum_{s \in \calS(M,\omega)} J_U(M,s), \label{E:spin_decomposition_omega}
   \end{align}
   and for the non-refined invariant we have
   \begin{align}
    J_U(W) &= \lambda(v_- 1_1)^\sigma \sum_{s \in \calS(M)} J_U(M,s) \label{E:spin_decomposition} ; 
   \end{align}
  \item If $p \equiv 2 \pmod 4$, then for every cohomology class $\varphi \in \calH(M)$ the scalar
   \begin{equation}
    J_U(M,\varphi) := \lambda(v_+ 1_0)^\sigma J_U(E,\delta^*(\varphi))
   \end{equation}
   is a topological invariant of the pair $(M,\varphi)$, for every relative cohomology class $\omega \in H^2(W,M;\Z/2\Z)$ we have
   \begin{align}
    J_U(W,\omega) &= \lambda(v_- 1_0)^\sigma \sum_{\varphi \in \calH(M,\omega)} J_U(M,\varphi), \label{E:cohomological_decomposition_omega}
   \end{align}
   and for the non-refined invariant we have
   \begin{align}
    J_U(W) &= \lambda(v_- 1_0)^\sigma \sum_{\varphi \in \calH(M)} J_U(M,\varphi). \label{E:cohomological_decomposition}
   \end{align}
 \end{enumerate}
\end{theorem}

\begin{remark}\label{R:to_be_precise} 
 Let us point out a few formal differences between the statement of Theorem~\ref{T:decomposition} and the one given in the Introduction. First of all, the signature $\sigma$ appearing here is the signature of the framed link $L$, so it coincides by definition with $\sigma(E)$, instead of $\sigma(W)$. Of course, this makes no difference, since trading $1$-han\-dles for $2$-han\-dles does not affect the signature. Indeed, every handle trade on a \mnfld{4} can be implemented by a \dmnsnl{5} cobordism, as explained in \cite[Remark~3.1.3]{KL01}, and the signature is a cobordism invariant. Furthermore, the signature renormalization appearing here involves the scalar $\lambda(v_- 1_1)$, if $p \equiv 0 \pmod 4$, or $\lambda(v_- 1_0)$, if $p \equiv 2 \pmod 4$, while in the Introduction we simply used $\lambda(v_-)$. However, we have $\lambda(v_-) = \lambda(v_- 1_0) + \lambda(v_- 1_1)$, and in both cases we are simply highlighting here the only non-vanishing summand.
\end{remark}

\begin{proof}[Proof of Theorem~\ref{T:decomposition}]
 Let us begin with point~$(1)$, so let us assume $p \equiv 0 \pmod 4$, and let us show that $J_U(M,s)$ is independent of the framed link $L$. Since by definition it is already invariant under Kirby II moves, meaning $2$-han\-dle slides, we only need to show that it is also invariant under Kirby I moves, meaning stabilization by $\partial \bbC P^2$ and $\partial \overline{\bbC P^2}$. In order to prove this, it is sufficient to notice that
 \[
  \lambda(v_+ 1_1) \lambda(v_- 1_1) = 1.
 \]
 This follows from Equation~\eqref{E:integral}, which combines with Equation~\eqref{E:graded_ribbon_1_2} to give
 \begin{align*}
  \lambda(v_+ 1_1) &= - \frac{\sqrt{2p} [p-1]!}{i^{p-1} \{ 1 \}^{p-1}} \frac{1-i}{\sqrt{p}} \frac{\{ -1 \}^{p-1}}{[p-1]!} \frac{q^{-\frac{(p+2)(p-1)}{2}} t}{2}
  = \frac{1-i}{\sqrt{2}} t^3,
 \end{align*}
 and with Equation~\eqref{E:graded_inverse_ribbon_1_2} to give
 \begin{align*}
  \lambda(v_- 1_1) &= - \frac{\sqrt{2p} [p-1]!}{i^{p-1} \{ 1 \}^{p-1}} \frac{1+i}{\sqrt{p}} \frac{\{ 1 \}^{p-1}}{[p-1]!} \frac{q^{\frac{(p+2)(p-1)}{2}} t^{-1}}{2}
  = \frac{1+i}{\sqrt{2}} t^{-3}.
 \end{align*}
 This establishes the first claim, because the signature of a Kirby diagram for $E \bcs (\bbC P^2 \smallsetminus \mathring{D}^4)$ is $\sigma+1$, while the signature of a Kirby diagram for $E \bcs (\overline{\bbC P^2} \smallsetminus \mathring{D}^4)$ is $\sigma-1$.

 Now, thanks to Proposition~\ref{P:general_decomposition}, since $U$ is factorizable, in order to conclude the proof of point~$(1)$, we simply need to assume that $W$ is geometrically simply connected, and to show that 
 \[
  J_U(W,\omega) = 0
 \]
 for every $\omega \in H^2(W,M;\Z/2\Z)$ which is not of the form $\omega = w_2(W,s)$ for some spin structure $s \in \calS(M)$. In other words, we need to show that 
 \[
  J_U(W,\omega) = 0
 \]
 whenever $\omega$ is not a characteristic sublink of $L$. Indeed, since for every other $\omega \in H^2(W,M;\Z/2\Z)$ there exists exactly one spin structure $s \in \calS(M)$ satisfying $w_2(W,s) = \omega$, this would imply Equation~\eqref{E:spin_decomposition_omega}, and Equation~\eqref{E:spin_decomposition} would follow from
 \[
  J_U(W) = \sum_{\omega \in H^2(W,M;\Z/2\Z)} J_U(W,\omega).
 \]
 Therefore, let us follow the structure of the proof of \cite[Theorem~III.3]{B92}, and let us fix an $\omega$ that is not a characteristic sublink. This means there exists a component $L_i \subset L$ such that
 \[
  \lk (L_i,\omega) \not\equiv \lk(L_i,L_i) \pmod 2.
 \]
 First of all, we claim that we can suppose that $L_i$ is unknotted. Indeed, by performing Kirby II moves every self crossing of $L_i$ can be changed as follows, without changing the invariant:
 \begin{align*}
  \pic{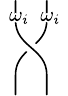} \rightsquigarrow
  \pic{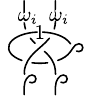}
 \end{align*}
 Therefore, we have the following configuration:
 \[
  \pic{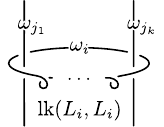}
 \]
 In particular, we have two cases:
 \begin{enumerate}
  \item If $L_i$ does not belong to the characteristic sublink, then $\omega_i = 0$. This means we have two further sub-cases:
   \begin{enumerate}
    \item If $\lk(L_i,L_i) \equiv 0 \pmod 2$, then $\omega_{j_1} + \ldots + \omega_{j_k} = 1$.
    \item If $\lk(L_i,L_i) \equiv 1 \pmod 2$, then $\omega_{j_1} + \ldots + \omega_{j_k} = 0$.
   \end{enumerate}
  \item If $L_i$ belongs to the characteristic sublink, then $\omega_i = 1$, and therefore $\omega_{j_1} + \ldots + \omega_{j_k} = 1$.
 \end{enumerate}

 Let us define a $\Z/2\Z$-grading on $U$ given by the algebra morphism 
 \[
  \deg : U \to \Z/2\Z
 \]
 determined by
 \begin{align*}
  \deg(E) &= 1, & \deg(F) &= 0, & \deg(K) &= 1.
 \end{align*}
 Notice that this is well-defined because all the relations of $U$ are homogeneous. Notice also that
 \[
  \deg(x) = 1 \Rightarrow \lambda(x) = 0,
 \]
 thanks to Equation~\eqref{E:integral}. However
 \begin{enumerate}
  \item $\deg(v_+ 1_0) = 1$, thanks to Equation~\eqref{E:graded_ribbon_0_2};
  \item $\deg(v_+ 1_1) = 0$, thanks to Equation~\eqref{E:graded_ribbon_1_2};
  \item $(\deg(S(M'_+) 1_0),\deg(M''_+ 1_1)) = (1,0)$, thanks to Equation~\eqref{E:graded_copairing_0_1};
  \item $(\deg(S(M'_+) 1_0),\deg(M''_+ 1_0)) = (0,0)$, thanks to Equation~\eqref{E:graded_copairing_0_0};
  \item $(\deg(S(M'_+) 1_1),\deg(M''_+ 1_1)) = (1,1)$, thanks to Equation~\eqref{E:graded_copairing_1_1}.
 \end{enumerate}
 This implies immediately that
 \begin{align*}
  \lambda(S(M'_+) 1_0 v_+^{2n} 1_0) M''_+ 1_1 &= 0, \\*
  \lambda(S(M'_+) 1_0 v_+^{2n+1} 1_0) M''_+ 1_0 &= 0, \\*
  \lambda(S(M'_+) 1_1 v_+^n 1_1) M''_+ 1_1 &= 0,
 \end{align*}
 which in particular means that
 \begin{align*}
  \lambda(S(M'_+) v_+^{2n} 1_0) (M''_+)_{(1)} 1_{\omega_{j_1}} \otimes \ldots \otimes (M''_+)_{(k)} 1_{\omega_{j_k}} &= 0, \\*
  \lambda(S(M'_+) v_+^{2n+1} 1_0) (M''_+)_{(1)} 1_{\omega_{j_1}} \otimes \ldots \otimes (M''_+)_{(k)} 1_{\omega_{j_k}} &= 0, \\*
  \lambda(S(M'_+) v_+^n 1_1) (M''_+)_{(1)} 1_{\omega_{j_1}} \otimes \ldots \otimes (M''_+)_{(k)} 1_{\omega_{j_k}} &= 0.
 \end{align*}
 This concludes the proof of point~$(1)$.

 The proof of point~$(2)$ is completely analogous, but we will briefly explain what should be adapted from the previous one, so let us assume $p \equiv 2 \pmod 4$. This time, we have
 \[
  \lambda(v_+ 1_0) \lambda(v_- 1_0) = 1.
 \]
 This follows from Equation~\eqref{E:integral}, which combines with Equation~\eqref{E:graded_ribbon_0_1} to give
 \begin{align*}
  \lambda(v_+ 1_0) &= \frac{\sqrt{2p} [p-1]!}{i^{p-1} \{ 1 \}^{p-1}} \frac{1-i}{\sqrt{p}} \frac{\{ -1 \}^{p-1}}{[p-1]!} \frac{q^{-\frac{(p+2)(p-1)}{2}} t}{2}
  = - \frac{1-i}{\sqrt{2}} t^3,
 \end{align*}
 and with Equation~\eqref{E:graded_inverse_ribbon_0_1} to give
 \begin{align*}
  \lambda(v_- 1_0) &= \frac{\sqrt{2p} [p-1]!}{i^{p-1} \{ 1 \}^{p-1}} \frac{1+i}{\sqrt{p}} \frac{\{ 1 \}^{p-1}}{[p-1]!} \frac{q^{\frac{(p+2)(p-1)}{2}} t^{-1}}{2}
  = - \frac{1+i}{\sqrt{2}} t^{-3}.
 \end{align*}
 This establishes the first claim.

 Now, in order to conclude, we need to assume that $W$ is geometrically simply connected, and to show that 
 \[
  J_U(W,\omega) = 0
 \]
 for every $\omega \in H^2(W,M;\Z/2\Z)$ which is not an even sublink of $L$. For such an $\omega$, there exists a component $L_i \subset L$ such that
 \[
  \lk (L_i,\omega) \not\equiv 0 \pmod 2,
 \]
 and again we can suppose that $L_i$ is unknotted. In particular, we have two cases:
 \begin{enumerate}
  \item If $L_i$ does not belong to the even sublink, then $\omega_i = 0$, and therefore $\omega_{j_1} + \ldots + \omega_{j_k} = 1$.
  \item If $L_i$ belongs to the even sublink, then $\omega_i = 1$. This means we have two further sub-cases:
   \begin{enumerate}
    \item If $\lk(L_i,L_i) \equiv 0 \pmod 2$, then $\omega_{j_1} + \ldots + \omega_{j_k} = 1$.
    \item If $\lk(L_i,L_i) \equiv 1 \pmod 2$, then $\omega_{j_1} + \ldots + \omega_{j_k} = 0$.
   \end{enumerate}
 \end{enumerate}
 However
 \begin{enumerate}
  \item $\deg(v_+ 1_0) = 0$, thanks to Equation~\eqref{E:graded_ribbon_0_1};
  \item $\deg(v_+ 1_1) = 1$, thanks to Equation~\eqref{E:graded_ribbon_1_1};
  \item $(\deg(S(M'_+) 1_0),\deg(M''_+ 1_1)) = (1,0)$, thanks to Equation~\eqref{E:graded_copairing_0_1};
  \item $(\deg(S(M'_+) 1_1),\deg(M''_+ 1_1)) = (1,1)$, thanks to Equation~\eqref{E:graded_copairing_1_1};
  \item $(\deg(S(M'_+) 1_1),\deg(M''_+ 1_0)) = (0,1)$, thanks to Equation~\eqref{E:graded_copairing_1_0}.
 \end{enumerate}
 This implies immediately that
 \begin{align*}
  \lambda(S(M'_+) 1_0 v_+^n 1_0) M''_+ 1_1 &= 0, \\*
  \lambda(S(M'_+) 1_1 v_+^{2n} 1_1) M''_+ 1_1 &= 0, \\*
  \lambda(S(M'_+) 1_1 v_+^{2n+1} 1_1) M''_+ 1_0 &= 0,
 \end{align*} 
 which in particular means that
 \begin{align*}
  \lambda(S(M'_+) v_+^n 1_0) (M''_+)_{(1)} 1_{\omega_{j_1}} \otimes \ldots \otimes (M''_+)_{(k)} 1_{\omega_{j_k}} &= 0, \\*
  \lambda(S(M'_+) v_+^{2n} 1_1) (M''_+)_{(1)} 1_{\omega_{j_1}} \otimes \ldots \otimes (M''_+)_{(k)} 1_{\omega_{j_k}} &= 0, \\*
  \lambda(S(M'_+) v_+^{2n+1} 1_1) (M''_+)_{(1)} 1_{\omega_{j_1}} \otimes \ldots \otimes (M''_+)_{(k)} 1_{\omega_{j_k}} &= 0. \qedhere
 \end{align*}
\end{proof}

\appendix

\section{Rescaling the invariant}\label{A:rescaling}

In this appendix, we explain the behavior of the invariant of \cite{BD21} under the operation of rescaling the integral and the cointegral. In order to do this, let us consider a unimodular ribbon category $\calC$ with end
\[
 \calE = \int_{X \in \calC} X \otimes X^*,
\]
as in \cite[Section~5.1]{BD21} (for instance, we could consider $\calC = \mods{H}$ for a unimodular ribbon Hopf algebra $H$, in which case $\calE = \myuline{H}$ would be the adjoint representation). Suppose that $\lambda : \calE \to \one$ is a two-sided integral, and that $\Lambda : \one \to \calE$ is a two-sided cointegral satisfying $\lambda \circ \Lambda = 1$. Let us denote by $J_\calC$ the associated scalar invariant of \dmnsnl{4} \hndlbds{2}, given by the restriction of the functor of \cite[Theorem~1.1]{BD21} to endomorphisms of the tensor unit of $\RHB$. Let us also denote by $J_{\calC,\xi}$ the scalar invariant obtained from $J_\calC$ by replacing $\lambda$ and $\Lambda$ with $\xi \lambda$ and $\xi^{-1} \Lambda$ respectively, for some invertible scalar $\xi \in \Bbbk^\times$.

\begin{lemma}\label{L:rescaling}
 For every \dmnsnl{4} \hndlbd{2} $W$ we have
 \begin{equation}\label{E:rescaling}
  J_{\calC,\xi}(W) = \xi^{\chi(W)-1} J_\calC(W),
 \end{equation}
 where $\chi(W)$ is the Euler characteristic of $W$.
\end{lemma}

\begin{proof}
 Notice that the space $C_k(W)$ of cellular $k$-chains of $W$ is generated by the set of cores $D^k \times \{ 0 \}$ of $k$-han\-dles $D^k \times D^{4-k}$. In particular, we have
 \begin{center} 
 \begin{tikzpicture}[descr/.style={fill=white}]
  \node (P3) at (-3,0) {$C_3(W)$};
  \node (P2) at (-1,0) {$C_2(W)$};
  \node (P1) at (1,0) {$C_1(W)$};
  \node (P0) at (3,0) {$C_0(W)$};
  \node (Q3) at (-3,0.5) {\rotatebox{90}{$=$}};  
  \node (Q0) at (3,0.5) {\rotatebox{90}{$=$}};
  \node (R3) at (-3,1) {$0$};  
  \node (R0) at (3,1) {$\Z$};
  \draw
  (P3) edge[->] node[above] {$0$} (P2)
  (P2) edge[->] node[above] {$\partial$} (P1)
  (P1) edge[->] node[above] {$0$} (P0);
 \end{tikzpicture}
 \end{center} 
 Therefore
 \begin{align*}
  H_2(W) &= \ker(\partial), &
  H_1(W) &= \coker(\partial).
 \end{align*}
 This implies that
 \[
  \rank(C_2(W))-\rank(C_1(W)) = \rank(\ker(\partial))-\rank(\coker(\partial)) = \chi(W)-1.
 \]
 Since 
 \[
  J_{\calC,\xi}(W) = \xi^{\rank(C_2(W))-\rank(C_1(W))} J_\calC(W),
 \]
 we proved the claim.
\end{proof}

Notice that Equation~\eqref{E:rescaling} behaves well under boundary connected sum, since
\[
 \chi(W \bcs W') = \chi(W) + \chi(W') - \chi(D^3) =  \chi(W) + \chi(W') - 1.
\]
Thanks to Lemma~\ref{L:rescaling}, the normalization of Equations~\eqref{E:small_integral} \& \eqref{E:small_cointegral} simply amounts to multiplying the invariant of $W$ computed in \cite[Corollary~9.3]{BD21} by $i^{(p-1)(\chi(W)-1)}$ (notice that $r = 2p$ in our current notation). 
%
%
Therefore, under these conventions, we have
\begin{align*}
 J_{\bar{U}}((S^2 \times S^2) \smallsetminus \mathring{D}^4) &= 1, \\*
 J_{\bar{U}}((S^2 \ttimes S^2) \smallsetminus \mathring{D}^4) &= 
 \begin{cases}
  1 & \mbox{ if } p \equiv 2 \pmod 4, \\
  0 & \mbox{ if } p \equiv 0 \pmod 4.
 \end{cases}
\end{align*}

\section{Ribbon element and copairing}

Let us compute the ribbon element $v_+ \in U$.

\begin{lemma}\label{L:ribbon}
 The ribbon element and its inverse are given by
 \begin{align*}
 v_+ &= \frac{1-i}{2\sqrt{p}} 
 \sum_{a=0}^{p-1} \sum_{b=0}^{2p-1} \frac{\{ -1 \}^a}{[a]!}
 q^{-\frac{(a+3)a}{2}} t^{(b+p+1)^2} F^a E^a K^{-a+b}, \\*
 v_- &= \frac{1+i}{2\sqrt{p}} 
 \sum_{a=0}^{p-1} \sum_{b=0}^{2p-1} \frac{\{ 1 \}^a}{[a]!} 
 q^{\frac{(a+3)a}{2}} t^{-(b+p-1)^2} F^a E^a K^{a+b}.
\end{align*}
\end{lemma}

\begin{proof}
 These formulas can be obtained by computing the Drinfeld element
 \begin{align*}
  u &= \frac{1}{4p} \sum_{a=0}^{p-1} \sum_{b,c=0}^{4p-1} \frac{\{ 1 \}^a}{[a]!}
  q^{\frac{a(a-1)}{2}} t^{-bc} S(k^c F^a) k^b E^a  \\*
  &= \frac{1}{4p} \sum_{a=0}^{p-1} \sum_{b,c=0}^{4p-1} (-1)^a \frac{\{ 1 \}^a}{[a]!}
  q^{-\frac{(a+3)a}{2}} t^{-bc} F^a k^{2a+b-c} E^a \\*
  &= \frac{1}{4p} \sum_{a=0}^{p-1} \sum_{b,c=0}^{4p-1} (-1)^a \frac{\{ 1 \}^a}{[a]!}
  q^{-\frac{(a+3)a}{2}} t^{2a(2a+b-c)-bc} F^a E^a k^{2a+b-c}
 \end{align*}
 and its inverse
 \begin{align*}
  u^{-1} &= \frac{1}{4p} \sum_{a=0}^{p-1} \sum_{b,c=0}^{4p-1} \frac{\{ 1 \}^a}{[a]!}
  q^{\frac{a(a-1)}{2}} t^{-bc} k^c F^a S^2(k^b E^a) \\*
  &= \frac{1}{4p} \sum_{a=0}^{p-1} \sum_{b,c=0}^{4p-1} \frac{\{ 1 \}^a}{[a]!}
  q^{\frac{(a+3)a}{2}} t^{-bc} k^c F^a k^b E^a \\*
  &= \frac{1}{4p} \sum_{a=0}^{p-1} \sum_{b,c=0}^{4p-1} \frac{\{ 1 \}^a}{[a]!}
  q^{\frac{(a+3)a}{2}} t^{2ab-bc} F^a E^a k^{b+c}.
 \end{align*}
 For what concerns $u$, if we set $d = 4a+b-c$ and $n = 2a+b$, we obtain
 \begin{align*}
  u 
  &= \frac{1}{4p} \sum_{a=0}^{p-1} \sum_{b,c=0}^{4p-1} \left( \sum_{n=0}^{4p-1} t^{-n^2+dn} \right) \frac{\{ -1 \}^a}{[a]!}
  q^{-\frac{(a+3)a}{2}} F^a E^a k^{-2a+d}.
 \end{align*}
 Similarly, for what concerns $u^{-1}$, if we set $d = -2a+b+c$ and $n = b$, we obtain
 \begin{align*}
  u^{-1} 
  &= \frac{1}{4p} \sum_{a=0}^{p-1} \left( \sum_{n=0}^{4p-1} t^{n^2+dn} \right) \frac{\{ 1 \}^a}{[a]!}
  q^{\frac{(a+3)a}{2}} F^a E^a k^{2a+d}.
 \end{align*}
 Since
 \begin{align*}
  \sum_{n=0}^{4p-1} t^{\pm n^2 + dn} &=
  \begin{cases}
   2\sqrt{p}(1 \pm i)t^{\mp \frac{d^2}{4}} & \mbox{ if } d \equiv 0 \pmod 2, \\
   0 & \mbox{ if } d \equiv 1 \pmod 2,
  \end{cases}
 \end{align*}
 this implies
 \begin{align*}
  u &= \frac{1-i}{2\sqrt{p}} \sum_{a=0}^{p-1} \sum_{b=0}^{2p-1} \frac{\{ -1 \}^a}{[a]!}
  q^{-\frac{(a+3)a}{2}} t^{b^2} F^a E^a k^{-2a+2b},
 \end{align*}
 and similarly
 \begin{align*}
  u^{-1} &= \frac{1+i}{2\sqrt{p}} \sum_{a=0}^{p-1} \sum_{b=0}^{2p-1} \frac{\{ 1 \}^a}{[a]!} q^{\frac{(a+3)a}{2}} t^{-b^2} F^a E^a k^{2a+2b}.
 \end{align*}
 The claim now follows from $v_+ = uK^{p-1}$ and $v_- = u^{-1}K^{p+1}$. 
\end{proof}

Let us compute the M-matrix $M_+ \in U \otimes U$.

\begin{lemma}\label{L:monodromy}
 The M-matrix and its inverse are given by
 \begin{align*}
  M_+ &= \frac{1}{2p} 
  \sum_{a,b=0}^{p-1} \sum_{c,d=0}^{2p-1} \frac{\{ 1 \}^{a+b}}{[a]![b]!} q^{\frac{a(a-1)+b(b-1)}{2} - 2b^2 - cd} K^{-b+c} F^b E^a \otimes K^{b+d} E^b F^a, \\*
  M_- &= \frac{1}{2p} 
  \sum_{a,b=0}^{p-1} \sum_{c,d=0}^{2p-1} \frac{\{ -1 \}^{a+b}}{[a]![b]!} q^{-\frac{a(a-1)+b(b-1)}{2} + 2b^2 + cd} E^a F^b K^{-b+c} \otimes F^a E^b K^{b+d}.
 \end{align*}
\end{lemma}

\begin{proof}
 We have
 \begin{align*}
  M_+ 
  &= \frac{1}{16p^2} \sum_{a,b=0}^{p-1} \sum_{c,d,e,f=0}^{4p-1} \frac{\{ 1 \}^{a+b}}{[a]![b]!} \\*
  &\hspace*{\parindent} q^{\frac{a(a-1)+b(b-1)}{2}} t^{-cd-ef} k^f F^b k^c E^a \otimes k^e E^b k^d F^a \\
  &= \frac{1}{16p^2} \sum_{a,b=0}^{p-1} \sum_{c,d,e,f=0}^{4p-1} \frac{\{ 1 \}^{a+b}}{[a]![b]!} \\*
  &\hspace*{\parindent} q^{\frac{a(a-1)+b(b-1)}{2}} t^{-cd-ef+2b(c-d)} k^{c+f} F^b E^a \otimes k^{d+e} E^b F^a \\
  &= \frac{1}{16p^2} \sum_{a,b=0}^{p-1} \sum_{e,f,g,h=0}^{4p-1} \frac{\{ 1 \}^{a+b}}{[a]![b]!} \\*
  &\hspace*{\parindent} q^{\frac{a(a-1)+b(b-1)}{2}} t^{-(f-g)(e-h)-ef-2b(f-g-e+h)} k^g F^b E^a \otimes k^h E^b F^a \\
  &= \frac{1}{4p} \sum_{a,b=0}^{p-1} \sum_{e,g,h=0}^{4p-1} \left( \frac{1}{4p} \sum_{f=0}^{4p-1} t^{-f(2b+2e-h)} \right) \frac{\{ 1 \}^{a+b}}{[a]![b]!} \\*
  &\hspace*{\parindent} q^{\frac{a(a-1)+b(b-1)}{2}} t^{(2b+g)(e-h)+2bg} k^g F^b E^a \otimes k^h E^b F^a \\
  &= \frac{1}{4p} \sum_{a,b=0}^{p-1} \sum_{e,g=0}^{4p-1} \frac{\{ 1 \}^{a+b}}{[a]![b]!} \\*
  &\hspace*{\parindent} q^{\frac{a(a-1)+b(b-1)}{2}} t^{-(2b+g)(2b+e)+2bg} k^g F^b E^a \otimes k^{2b+2e} E^b F^a \\
  &= \frac{1}{4p} \sum_{a,b=0}^{p-1} \sum_{e=0}^{2p-1} \sum_{g=0}^{4p-1} \frac{\{ 1 \}^{a+b}}{[a]![b]!} (1+(-1)^g) \\*
  &\hspace*{\parindent} q^{\frac{a(a-1)+b(b-1)}{2}} t^{-(2b+g)(2b+e)+2bg} k^g F^b E^a \otimes k^{2b+2e} E^b F^a \\
  &= \frac{1}{2p} \sum_{a,b=0}^{p-1} \sum_{e,g=0}^{2p-1} \frac{\{ 1 \}^{a+b}}{[a]![b]!} \\*
  &\hspace*{\parindent} q^{\frac{a(a-1)+b(b-1)}{2}} t^{-(2b+2g)(2b+e)+4bg} k^{2g} F^b E^a \otimes k^{2b+2e} E^b F^a.
 \end{align*}
 Similarly, we have
 \begin{align*}
  M_-
  &= \frac{1}{16p^2} \sum_{a,b=0}^{p-1} \sum_{c,d,e,f=0}^{4p-1} \frac{\{ -1 \}^{a+b}}{[a]![b]!} \\*
  &\hspace*{\parindent} q^{-\frac{a(a-1)+b(b-1)}{2}} t^{cd+ef}  E^a k^c F^b k^f \otimes F^a k^d E^b k^e \\
  &= \frac{1}{16p^2} \sum_{a,b=0}^{p-1} \sum_{c,d,e,f=0}^{4p-1} \frac{\{ -1 \}^{a+b}}{[a]![b]!} \\*
  &\hspace*{\parindent} q^{-\frac{a(a-1)+b(b-1)}{2}} t^{cd+ef-2b(c-d)} E^a F^b k^{c+f} \otimes F^a E^b k^{d+e} \\
  &= \frac{1}{16p^2} \sum_{a,b=0}^{p-1} \sum_{e,f,g,h=0}^{4p-1} \frac{\{ -1 \}^{a+b}}{[a]![b]!} \\*
  &\hspace*{\parindent} q^{-\frac{a(a-1)+b(b-1)}{2}} t^{(f-g)(e-h)+ef+2b(f-g-e+h)} E^a F^b k^g \otimes F^a E^b k^h \\
  &= \frac{1}{4p} \sum_{a,b=0}^{p-1} \sum_{e,g,h=0}^{4p-1} \left( \frac{1}{4p} \sum_{f=0}^{4p-1} t^{f(2b+2e-h)} \right) \frac{\{ -1 \}^{a+b}}{[a]![b]!} \\*
  &\hspace*{\parindent} q^{-\frac{a(a-1)+b(b-1)}{2}} t^{-(2b+g)(e-h)-2bg} E^a F^b k^g \otimes F^a E^b k^h \\
  &= \frac{1}{4p} \sum_{a,b=0}^{p-1} \sum_{e,g=0}^{4p-1} \frac{\{ -1 \}^{a+b}}{[a]![b]!} \\*
  &\hspace*{\parindent} q^{-\frac{a(a-1)+b(b-1)}{2}} t^{(2b+g)(2b+e)-2bg} E^a F^b k^g \otimes F^a E^b k^{2b+2e} \\
  &= \frac{1}{4p} \sum_{a,b=0}^{p-1} \sum_{e=0}^{2p-1} \sum_{g=0}^{4p-1} \frac{\{ -1 \}^{a+b}}{[a]![b]!} (1+(-1)^g) \\*
  &\hspace*{\parindent} q^{-\frac{a(a-1)+b(b-1)}{2}} t^{(2b+g)(2b+e)-2bg} E^a F^b k^g \otimes F^a E^b k^{2b+2e} \\
  &= \frac{1}{2p} \sum_{a,b=0}^{p-1} \sum_{e,g=0}^{2p-1} \frac{\{ -1 \}^{a+b}}{[a]![b]!} \\*
  &\hspace*{\parindent} q^{-\frac{a(a-1)+b(b-1)}{2}} t^{(2b+2g)(2b+e)-4bg} E^a F^b k^{2g} \otimes F^a E^b k^{2b+2e}.
 \end{align*}
 The claim now follows by setting $g = -b+c$ and $e = d$. 
\end{proof}

Let us compute the copairing 
\[
 w_+ := \myuline{w_+}(1) \in U \otimes U.
\]

\begin{lemma}\label{L:copairing}
 The copairing is given by
 \begin{align*}
  w_+ 
  &= \frac{1}{2p} 
  \sum_{a,b=0}^{p-1} \sum_{c,d=0}^{2p-1} \frac{\{ -1 \}^{a+b}}{[a]![b]!} q^{-\frac{a(a-1)+b(b-1)}{2} - 2b + cd} E^a F^b K^{a+c} \otimes E^b F^a K^{b+d}.
 \end{align*}
\end{lemma}

\begin{proof}
 We have
 \begin{align*}
  w_+ 
  &= \frac{1}{2p} 
  \sum_{a,b=0}^{p-1} \sum_{c,d=0}^{2p-1} \frac{\{ 1 \}^{a+b}}{[a]![b]!} q^{\frac{a(a-1)+b(b-1)}{2} - 2b^2 - cd} S(K^{-b+c} F^b E^a) \otimes K^{b+d} E^b F^a \\*
  &= \frac{1}{2p} 
  \sum_{a,b=0}^{p-1} \sum_{c,d=0}^{2p-1} \frac{\{ -1 \}^{a+b}}{[a]![b]!} q^{\frac{a(a-1)+b(b-1)}{2} - (a-b)(a-b-1) - 2b^2 - cd} \\*
  &\hspace*{\parindent} E^a F^b K^{-a+2b-c} \otimes K^{b+d} E^b F^a \\*
  &= \frac{1}{2p} 
  \sum_{a,b=0}^{p-1} \sum_{c,d=0}^{2p-1} \frac{\{ -1 \}^{a+b}}{[a]![b]!} q^{\frac{a(a-1)+b(b-1)}{2} - (a-b)(a-b-1) - 2(a-b)(b+d) - 2b^2 - cd} \\*
  &\hspace*{\parindent} E^a F^b K^{-a+2b-c} \otimes E^b F^a K^{b+d} \\*
  &= \frac{1}{2p} 
  \sum_{a,b=0}^{p-1} \sum_{c,d=0}^{2p-1} \frac{\{ -1 \}^{a+b}}{[a]![b]!} q^{\frac{a(a-1)+b(b-1)}{2} - (a-b)(a-b-1) - 2(a-b)(b+d) - 2b^2 + (2a-2b+c)d} \\*
  &\hspace*{\parindent} E^a F^b K^{a+c} \otimes E^b F^a K^{b+d} \\*
  &= \frac{1}{2p} 
  \sum_{a,b=0}^{p-1} \sum_{c,d=0}^{2p-1} \frac{\{ -1 \}^{a+b}}{[a]![b]!} q^{-\frac{a(a-1)+b(b-1)}{2} - 2b + cd} E^a F^b K^{a+c} \otimes E^b F^a K^{b+d}. \qedhere
 \end{align*}
\end{proof}

From now on, let us focus on the case $p \equiv 0 \pmod 2$. Let us compute the graded ribbon and inverse ribbon elements
\begin{align*}
 v_\alpha &:= v_+ 1_\alpha \in U 1_\alpha, &
 v_\alpha^{-1} &:= v_- 1_\alpha \in U 1_\alpha.
\end{align*}

\begin{lemma}\label{L:graded_ribbon}
 The graded ribbon and inverse ribbon elements are given for $p \equiv 2 \pmod 4$ by
 \begin{align}
  v_+  1_0
  &= \frac{1-i}{\sqrt{p}} 
  \sum_{a=0}^{p-1} \sum_{b=0}^{p'-1} \frac{\{ -1 \}^a}{[a]!}
  q^{-\frac{(a+3)a}{2}} t^{(2b-1)^2} F^a E^a K^{-a-2b} 1_0, \label{E:graded_ribbon_0_1} \\*
  v_+  1_1
  &= - \frac{1-i}{\sqrt{p}} 
  \sum_{a=0}^{p-1} \sum_{b=0}^{p'-1} \frac{\{ -1 \}^a}{[a]!}
  q^{-\frac{(a+3)a}{2}} t^{(2b)^2} F^a E^a K^{-a-2b-1} 1_1, \label{E:graded_ribbon_1_1} \\*
  v_-  1_0
  &= \frac{1+i}{\sqrt{p}} 
  \sum_{a=0}^{p-1} \sum_{b=0}^{p'-1} \frac{\{ 1 \}^a}{[a]!}
  q^{\frac{(a+3)a}{2}} t^{-(2b-1)^2} F^a E^a K^{a+2b} 1_0, \label{E:graded_inverse_ribbon_0_1} \\*
  v_-  1_1
  &= - \frac{1+i}{\sqrt{p}} 
  \sum_{a=0}^{p-1} \sum_{b=0}^{p'-1} \frac{\{ 1 \}^a}{[a]!}
  q^{\frac{(a+3)a}{2}} t^{-(2b)^2} F^a E^a K^{a+2b+1} 1_1, \label{E:graded_inverse_ribbon_1_1}
 \end{align}
 and for $p \equiv 0 \pmod 4$ by
 \begin{align}
  v_+ 1_0
  &= \frac{1-i}{\sqrt{p}} 
  \sum_{a=0}^{p-1} \sum_{b=0}^{p'-1} \frac{\{ -1 \}^a}{[a]!} 
  q^{-\frac{(a+3)a}{2}} t^{(2b)^2} F^a E^a K^{-a-2b-1} 1_0, \label{E:graded_ribbon_0_2} \\*
  v_+ 1_1
  &= - \frac{1-i}{\sqrt{p}} 
  \sum_{a=0}^{p-1} \sum_{b=0}^{p'-1} \frac{\{ -1 \}^a}{[a]!} 
  q^{-\frac{(a+3)a}{2}} t^{(2b-1)^2} F^a E^a K^{-a-2b} 1_1. \label{E:graded_ribbon_1_2} \\*
  v_- 1_0
  &= \frac{1+i}{\sqrt{p}} 
  \sum_{a=0}^{p-1} \sum_{b=0}^{p'-1} \frac{\{ 1 \}^a}{[a]!} 
  q^{\frac{(a+3)a}{2}} t^{-(2b)^2} F^a E^a K^{a+2b+1} 1_0, \label{E:graded_inverse_ribbon_0_2} \\*
  v_- 1_1
  &= - \frac{1+i}{\sqrt{p}} 
  \sum_{a=0}^{p-1} \sum_{b=0}^{p'-1} \frac{\{ 1 \}^a}{[a]!} 
  q^{\frac{(a+3)a}{2}} t^{-(2b-1)^2} F^a E^a K^{a+2b} 1_1. \label{E:graded_inverse_ribbon_1_2}
 \end{align}
\end{lemma}

\begin{proof}
 We have
 \begin{align*}
  v_+ 1_0 &= \frac{1-i}{2\sqrt{p}} 
  \sum_{a=0}^{p-1} \sum_{b=0}^{2p-1} \frac{\{ -1 \}^a}{[a]!} i^p (-1)^{b+1}
  q^{-\frac{(a+3)a}{2}} t^{(b-1)^2} F^a E^a K^{-a-b} 1_0 \\*
   &= \frac{1-i}{2\sqrt{p}} 
  \sum_{a=0}^{p-1} \sum_{b=0}^{p-1} \frac{\{ -1 \}^a}{[a]!} 
  q^{-\frac{(a+3)a}{2}} t^{(b-1)^2} \left( 1 + i^p (-1)^{b+1} \right) F^a E^a K^{-a-b} 1_0, \\
  v_+ 1_1 &= \frac{1-i}{2\sqrt{p}} 
  \sum_{a=0}^{p-1} \sum_{b=0}^{2p-1} \frac{\{ -1 \}^a}{[a]!} i^p (-1)^{b+1}
  q^{-\frac{(a+3)a}{2}} t^{(b-1)^2} F^a E^a K^{-a-b} 1_1 \\*
  &= \frac{1-i}{2\sqrt{p}} 
  \sum_{a=0}^{p-1} \sum_{b=0}^{p-1} \frac{\{ -1 \}^a}{[a]!} 
   q^{-\frac{(a+3)a}{2}} t^{(b-1)^2} \left( - 1 + i^p (-1)^{b+1} \right) F^a E^a K^{-a-b} 1_1, \\ 
  v_- 1_0 &= \frac{1+i}{2\sqrt{p}} 
  \sum_{a=0}^{p-1} \sum_{b=0}^{2p-1} \frac{\{ 1 \}^a}{[a]!} i^{-p} (-1)^{b+1}
  q^{\frac{(a+3)a}{2}} t^{-(b-1)^2} F^a E^a K^{a+b} 1_0 \\*
   &= \frac{1+i}{2\sqrt{p}} 
  \sum_{a=0}^{p-1} \sum_{b=0}^{p-1} \frac{\{ 1 \}^a}{[a]!} 
  q^{\frac{(a+3)a}{2}} t^{-(b-1)^2} \left( 1 + i^p (-1)^{b+1} \right) F^a E^a K^{a+b} 1_0, \\
  v_+ 1_1 &= \frac{1+i}{2\sqrt{p}} 
  \sum_{a=0}^{p-1} \sum_{b=0}^{2p-1} \frac{\{ 1 \}^a}{[a]!} i^{-p} (-1)^{b+1}
  q^{\frac{(a+3)a}{2}} t^{-(b-1)^2} F^a E^a K^{a+b} 1_1 \\*
  &= \frac{1+i}{2\sqrt{p}} 
  \sum_{a=0}^{p-1} \sum_{b=0}^{p-1} \frac{\{ 1 \}^a}{[a]!} 
   q^{\frac{(a+3)a}{2}} t^{(b-1)^2} \left( - 1 + i^p (-1)^{b+1} \right) F^a E^a K^{a+b} 1_1. \qedhere
 \end{align*}
\end{proof}

Let us compute the graded copairing 
\[
 w_{\alpha,\beta} := w_+ (1_\alpha \otimes 1_\beta) \in U 1_\alpha \otimes U 1_\beta.
\]

\begin{lemma}\label{L:graded_copairing}
 The graded copairing is given by
 \begin{align}
  w_+(1_0 \otimes 1_0)
  &= \frac{1}{p'} 
  \sum_{a,b=0}^{p-1} \sum_{c,d=0}^{p'-1} \frac{\{ -1 \}^{a+b}}{[a]![b]!} q^{-\frac{a(a-1)+b(b-1)}{2} - 2b + 4cd} \nonumber \\*
  &\hspace*{\parindent} E^a F^b K^{a+2c} 1_0 \otimes E^b F^a K^{b+2d} 1_0,  \label{E:graded_copairing_0_0} \\*
  w_+(1_0 \otimes 1_1)
  &= \frac{1}{p'} 
  \sum_{a,b=0}^{p-1} \sum_{c,d=0}^{p'-1} \frac{\{ -1 \}^{a+b}}{[a]![b]!} q^{-\frac{a(a-1)+b(b-1)}{2} - 2b + 2d + 4cd} \nonumber \\*
  &\hspace*{\parindent} E^a F^b K^{a+2c+1} 1_0 \otimes E^b F^a K^{b+2d} 1_1, \label{E:graded_copairing_0_1} \\*
  w_+(1_1 \otimes 1_0)
  &= \frac{1}{p'} 
  \sum_{a,b=0}^{p-1} \sum_{c,d=0}^{p'-1} \frac{\{ -1 \}^{a+b}}{[a]![b]!} q^{-\frac{a(a-1)+b(b-1)}{2} - 2b + 2c + 4cd} \nonumber \\*
  &\hspace*{\parindent} E^a F^b K^{a+2c} 1_1 \otimes E^b F^a K^{b+2d+1} 1_0, \label{E:graded_copairing_1_0} \\*
  w_+(1_1 \otimes 1_1)
  &= \frac{1}{p'} 
  \sum_{a,b=0}^{p-1} \sum_{c,d=0}^{p'-1} \frac{\{ -1 \}^{a+b}}{[a]![b]!} q^{-\frac{a(a-1)+b(b-1)}{2} - 2b + 2(c+d) + 4cd} \nonumber \\*
  &\hspace*{\parindent} E^a F^b K^{a+2c+1} 1_1 \otimes E^b F^a K^{b+2d+1} 1_1. \label{E:graded_copairing_1_1} 
\end{align}
\end{lemma}

\begin{proof}
 We have
 \begin{align*}
  w_+(1_0 \otimes 1_0)
  &= \frac{1}{2p} 
  \sum_{a,b,c=0}^{p-1} \sum_{d=0}^{2p-1} \frac{\{ -1 \}^{a+b}}{[a]![b]!} q^{-\frac{a(a-1)+b(b-1)}{2} - 2b + cd} \\*
  &\hspace*{\parindent} (1+(-1)^d) E^a F^b K^{a+c} 1_0 \otimes E^b F^a K^{b+d} 1_0 \\*
  &= \frac{1}{p} 
  \sum_{a,b,c,d=0}^{p-1} \frac{\{ -1 \}^{a+b}}{[a]![b]!} q^{-\frac{a(a-1)+b(b-1)}{2} - 2b + 2cd} \\*
  &\hspace*{\parindent} E^a F^b K^{a+c} 1_0 \otimes E^b F^a K^{b+2d} 1_0 \\*
  &= \frac{1}{p} 
  \sum_{a,b,c=0}^{p-1} \sum_{d=0}^{p'-1} \frac{\{ -1 \}^{a+b}}{[a]![b]!} q^{-\frac{a(a-1)+b(b-1)}{2} - 2b + 2cd} \\*
  &\hspace*{\parindent} (1+(-1)^c) E^a F^b K^{a+c} 1_0 \otimes E^b F^a K^{b+2d} 1_0 \\*
  &= \frac{1}{p'} 
  \sum_{a,b=0}^{p-1} \sum_{c,d=0}^{p'-1} \frac{\{ -1 \}^{a+b}}{[a]![b]!} q^{-\frac{a(a-1)+b(b-1)}{2} - 2b + 4cd} \\*
  &\hspace*{\parindent} E^a F^b K^{a+2c} 1_0 \otimes E^b F^a K^{b+2d} 1_0, \\
  w_+(1_0 \otimes 1_1)
  &= \frac{1}{2p} 
  \sum_{a,b,c=0}^{p-1} \sum_{d=0}^{2p-1} \frac{\{ -1 \}^{a+b}}{[a]![b]!} q^{-\frac{a(a-1)+b(b-1)}{2} - 2b + cd} \\*
  &\hspace*{\parindent} (1+(-1)^d) E^a F^b K^{a+c} 1_0 \otimes E^b F^a K^{b+d} 1_1 \\*
  &= \frac{1}{p} 
  \sum_{a,b,c,d=0}^{p-1} \frac{\{ -1 \}^{a+b}}{[a]![b]!} q^{-\frac{a(a-1)+b(b-1)}{2} - 2b + 2cd} \\*
  &\hspace*{\parindent} E^a F^b K^{a+c} 1_0 \otimes E^b F^a K^{b+2d} 1_1 \\*
  &= \frac{1}{p} 
  \sum_{a,b,c=0}^{p-1} \sum_{d=0}^{p'-1} \frac{\{ -1 \}^{a+b}}{[a]![b]!} q^{-\frac{a(a-1)+b(b-1)}{2} - 2b + 2cd} \\*
  &\hspace*{\parindent} (1-(-1)^c) E^a F^b K^{a+c} 1_0 \otimes E^b F^a K^{b+2d} 1_1 \\*
  &= \frac{1}{p'} 
  \sum_{a,b=0}^{p-1} \sum_{c,d=0}^{p'-1} \frac{\{ -1 \}^{a+b}}{[a]![b]!} q^{-\frac{a(a-1)+b(b-1)}{2} - 2b + 2d + 4cd} \\*
  &\hspace*{\parindent} E^a F^b K^{a+2c+1} 1_0 \otimes E^b F^a K^{b+2d} 1_1, \\
  w_+(1_1 \otimes 1_0)
  &= \frac{1}{2p} 
  \sum_{a,b,c=0}^{p-1} \sum_{d=0}^{2p-1} \frac{\{ -1 \}^{a+b}}{[a]![b]!} q^{-\frac{a(a-1)+b(b-1)}{2} - 2b + cd} \\*
  &\hspace*{\parindent} (1-(-1)^d) E^a F^b K^{a+c} 1_1 \otimes E^b F^a K^{b+d} 1_0 \\*
  &= \frac{1}{p} 
  \sum_{a,b,c,d=0}^{p-1} \frac{\{ -1 \}^{a+b}}{[a]![b]!} q^{-\frac{a(a-1)+b(b-1)}{2} - 2b + c + 2cd} \\*
  &\hspace*{\parindent} E^a F^b K^{a+c} 1_1 \otimes E^b F^a K^{b+2d+1} 1_0 \\*
  &= \frac{1}{p} 
  \sum_{a,b,c=0}^{p-1} \sum_{d=0}^{p'-1} \frac{\{ -1 \}^{a+b}}{[a]![b]!} q^{-\frac{a(a-1)+b(b-1)}{2} - 2b + c + 2cd} \\*
  &\hspace*{\parindent} (1+(-1)^c) E^a F^b K^{a+c} 1_1 \otimes E^b F^a K^{b+2d+1} 1_0 \\*
  &= \frac{1}{p'} 
  \sum_{a,b=0}^{p-1} \sum_{c,d=0}^{p'-1} \frac{\{ -1 \}^{a+b}}{[a]![b]!} q^{-\frac{a(a-1)+b(b-1)}{2} - 2b + 2c + 4cd} \\*
  &\hspace*{\parindent} E^a F^b K^{a+2c} 1_1 \otimes E^b F^a K^{b+2d+1} 1_0, \\
  w_+(1_1 \otimes 1_1)
  &= \frac{1}{2p} 
  \sum_{a,b,c=0}^{p-1} \sum_{d=0}^{2p-1} \frac{\{ -1 \}^{a+b}}{[a]![b]!} q^{-\frac{a(a-1)+b(b-1)}{2} - 2b + cd} \\*
  &\hspace*{\parindent} (1-(-1)^d) E^a F^b K^{a+c} 1_1 \otimes E^b F^a K^{b+d} 1_1 \\*
  &= \frac{1}{p} 
  \sum_{a,b,c,d=0}^{p-1} \frac{\{ -1 \}^{a+b}}{[a]![b]!} q^{-\frac{a(a-1)+b(b-1)}{2} - 2b + c + 2cd} \\*
  &\hspace*{\parindent} E^a F^b K^{a+c} 1_1 \otimes E^b F^a K^{b+2d+1} 1_1 \\*
  &= \frac{1}{p} 
  \sum_{a,b,c=0}^{p-1} \sum_{d=0}^{p'-1} \frac{\{ -1 \}^{a+b}}{[a]![b]!} q^{-\frac{a(a-1)+b(b-1)}{2} - 2b + c + 2cd} \\*
  &\hspace*{\parindent} (1-(-1)^c) E^a F^b K^{a+c} 1_1 \otimes E^b F^a K^{b+2d+1} 1_1 \\*
  &= \frac{1}{p'} 
  \sum_{a,b=0}^{p-1} \sum_{c,d=0}^{p'-1} \frac{\{ -1 \}^{a+b}}{[a]![b]!} q^{-\frac{a(a-1)+b(b-1)}{2} - 2b + 2(c+d) + 4cd} \\*
  &\hspace*{\parindent} E^a F^b K^{a+2c+1} 1_1 \otimes E^b F^a K^{b+2d+1} 1_1. \qedhere
 \end{align*}
\end{proof}

\end{document}